\documentclass [11pt]{amsart}
\usepackage {amsmath, amssymb, amscd, mathrsfs, hyperref, slashed, enumerate, url}
\usepackage[text={6in,8.5in},centering,letterpaper,dvips]{geometry}
\usepackage[all,cmtip]{xy}

\newtheorem {theorem}{Theorem}[section]
\newtheorem {lemma}[theorem]{Lemma}
\newtheorem {proposition}[theorem]{Proposition}
\newtheorem {corollary}[theorem]{Corollary}
\newtheorem {conjecture}[theorem]{Conjecture}
\newtheorem {definition}[theorem]{Definition}

\theoremstyle{remark}
\newtheorem {remark}[theorem]{Remark}
\newtheorem {example}[theorem]{Example}

\DeclareFontFamily{U}{mathx}{\hyphenchar\font45}
\DeclareFontShape{U}{mathx}{m}{n}{
      <5> <6> <7> <8> <9> <10>
      <10.95> <12> <14.4> <17.28> <20.74> <24.88>
      mathx10
      }{}
\DeclareSymbolFont{mathx}{U}{mathx}{m}{n}
\DeclareFontSubstitution{U}{mathx}{m}{n}
\DeclareMathAccent{\widecheck}{0}{mathx}{"71}
\newcommand\eps{\varepsilon}
\def\polhk#1{\setbox0=\hbox{#1}{\ooalign{\hidewidth
    \lower1.5ex\hbox{`}\hidewidth\crcr\unhbox0}}}  
\def\Z {{\mathbb{Z}}}
\def\R {{\mathbb{R}}}
\def\C {{\mathbb{C}}}
\def\Q {{\mathbb{Q}}}
\def\Conf {\mathcal{C}}
\def\U {\mathcal{U}}
\def\tR{\tilde{\R}}
\def\Spin {\mathbb{S}}
\def\dirac {\slashed{\partial}}
\def\Dirac{\slashed{D}}
\def\tH{\tilde{H}}
\def\rp{\mathbb{RP}}
\def\cp{\mathbb{CP}}
\def\hp{\mathbb{HP}}
\def\F {\mathbb{F}_2}
\def\del {\partial}
\def\f {\mathbb{F}}
\def\t{\mathfrak{t}}

\def\ind{\operatorname{ind}}

\DeclareMathOperator{\im}{\operatorname{image}}
\def\Sq{\operatorname{Sq}}

\def\Inv{\operatorname{Inv}}
\def\To {\longrightarrow}
\def\su {\operatorname{SU}(2)}
\def\Ring {\mathcal{R}}

\def\mod {\operatorname{mod}}

\def\swf{\operatorname{SWF}}
\def\swfh{\mathit{SWFH}}
\def\iswfh{ ^{\infty}\! \swfh}
\def\pin {\operatorname{Pin}(2)}
\newcommand{\HM}{\mathit{HM}}
\newcommand{\HMto}{\widecheck{\mathit{HM}}}
\def\HF{\mathit{HF}}
\def\H {\mathbb{H}}
\def\iH{ ^{\infty}\! \tH}
\def\J {\mathcal{J}}
\def\csd{\mathit{CSD}}
\def\swfsmall {\text{swf}}
\def\G {\mathcal{G}}
\newcommand{\s}{\mathfrak{s}}
\newcommand{\pml}{p^\nu_\tau}

\newcommand{\vnt}{V^\nu_\tau}

\def\pt {\operatorname{pt}}
\def\irr{\operatorname{irr}}
\def\eps{\varepsilon}
\begin{document}

\title[Pin(2)-equivariant Seiberg-Witten Floer homology]{Pin(2)-equivariant Seiberg-Witten Floer homology and the Triangulation Conjecture}

\author[Ciprian Manolescu]{Ciprian Manolescu}
\thanks {The author was supported by NSF grant DMS-1104406.}
\address {Department of Mathematics, UCLA, 520 Portola Plaza\\ Los Angeles, CA 90095}
\email {cm@math.ucla.edu}

\begin{abstract}
We define $\pin$-equivariant Seiberg-Witten Floer homology for rational homology $3$-spheres equipped with a spin structure. The analogue of Fr{\o}yshov's correction term in this setting is an integer-valued invariant of homology cobordism whose mod $2$ reduction is the Rokhlin invariant. As an application, we show that there are no homology $3$-spheres $Y$ of Rokhlin invariant one such that $Y \#Y$ bounds an acyclic smooth $4$-manifold. By previous work of Galewski-Stern and Matumoto, this implies the existence of non-triangulable high-dimensional manifolds. 
\end {abstract}

\maketitle

\section{Introduction}

The existence of topological manifolds that do not admit combinatorial triangulations (that is, piecewise linear structures) has been known in dimension $\geq 5$ since the celebrated work of Kirby and Siebenmann \cite{KSbook}. Freedman \cite{Freedman} constructed such examples in dimension four. In dimensions $\leq 3$, every topological manifold has a unique piecewise-linear structure, by the older work of Rad\'o \cite{Rado} and Moise \cite{Moise}. 

A related question is whether topological manifolds admit simplicial triangulations. A simplicial triangulation is a homeomorphism to a locally finite simplicial complex; this complex does not have to be a piecewise linear manifold. A typical example of a simplicial triangulation that is not combinatorial can be obtained as follows: Take a non-trivial homology sphere $M$ (such as the Poincar\'e sphere), and form its double suspension $\Sigma^2 M$. By the double suspension theorem of Edwards \cite{Edwards, EdwardsICM} and Cannon \cite{Cannon}, $\Sigma^2 M$ is homeomorphic to a sphere. An arbitrary triangulation of $M$ induces one of $\Sigma^2 M$ that is not combinatorial, because the links of the cone points are not spheres. 

The Triangulation Conjecture (in dimension $n$) states that every $n$-dimensional topological manifold has a simplicial triangulation. The conjecture is true for $n \leq 3$, but false in dimension four: Using the properties of the Casson invariant \cite{Casson}, it can be shown that the Freedman $E_8$-manifold cannot be triangulated. 

In dimensions $\geq 5$, Galewski and Stern \cite{GS} and Matumoto \cite{Matumoto} reduced the Triangulation Conjecture to a problem in low-dimensional topology: They showed that the conjecture is true if and only if there exists a homology $3$-sphere $Y$ such that $Y$ has Rokhlin invariant one, and $Y$ is of order two in the homology cobordism group $\theta^H_3$. 

Let us recall the relevant definitions: The group $\theta^H_3$ is generated by equivalence classes of  integral homology $3$-spheres,  where $Y_0$ is equivalent to $Y_1$ if there exists a piecewise-linear (or, equivalently, a smooth) compact, oriented $4$-dimensional cobordism $W$ from $Y_0$ to $Y_1$, such that $H_*(W, Y_0; \Z) = H_*(W, Y_1; \Z)=0$. Addition in $\theta^H_3$ is given by connected sum. It is known that $\theta^H_3$ is infinite, and in fact infinitely generated \cite{FSorbifolds, FurutaHom, FSinstanton}. There is a distinguished map $\mu: \theta^H_3 \to \Z/2$, called the Rokhlin homomorphism \cite{Rokhlin, EellsKuiper}. More generally, one can associate a Rokhlin invariant $\mu(Y, \s) \in \tfrac{1}{8}\Z \pmod {2\Z}$ to any $3$-manifold $Y$ equipped with a spin structure $\s$. One takes an arbitrary compact, smooth, spin four-manifold $(W, \t)$ with boundary $(Y, \s)$ and sets
$$ \mu(Y, \s) = \frac{\sigma(W)}{8} \pmod {2\Z},$$
where $\sigma(W)$ denotes the signature of $W$. When $Y$ is an integral homology sphere, there is a unique spin structure $\s$ on $Y$, and $\sigma(W)$ is divisible by $8$; the Rokhlin homomorphism is defined by $\mu(Y) = \mu(Y, \s)$. 

The main result of this paper is:

\begin{theorem}
\label{thm:main}
To every rational homology $3$-sphere $Y$ equipped with a spin structure $\s$ we can associate an  invariant $\beta(Y, \s) \in \tfrac{1}{8}\Z$, with the following properties:
\begin{enumerate}
\item If $-Y$ denotes $Y$ with the orientation reversed, then $\beta(-Y, \s) = - \beta(Y, \s)$;
\item The mod $2$ reduction of $-\beta(Y, \s)$ is the generalized Rokhlin invariant $\mu(Y, \s)$;
\item Suppose that $W$ is a smooth, oriented, negative-definite cobordism from $Y_0$ to $Y_1$, and let $b_2(W)$ denote the second Betti number of $W$. If $W$ admits a spin structure $\t$, then 
$$\beta(Y_1, \t|_{Y_1}) \geq \beta(Y_0, \t|_{Y_0}) + \frac{1}{8} b_2(W).$$
\end{enumerate}
\end{theorem}

When $Y$ is an integral homology sphere, we simply write $\beta(Y)=\beta(Y, \s) \in \Z$ for the unique spin structure $\s$. We then have $\beta(Y) \equiv \mu(Y) \pmod 2$. The third property of $\beta$ mentioned in Theorem~\ref{thm:main} shows that $\beta$ is an invariant of homology cobordism. Together with the other two properties, this implies  the following:

\begin{corollary}
\label{cor:no2}
If $Y$ is a homology sphere of Rokhlin invariant one, then $Y \# Y$ is not homology cobordant to $S^3$.
\end{corollary} 

In view of the work of Galewski-Stern and Matumoto \cite{GS, Matumoto}, this disproves the Triangulation Conjecture in high dimensions:

\begin{corollary}
For every $n \geq 5$, there exists a closed $n$-dimensional topological manifold that does not admit a  simplicial triangulation.
\end{corollary} 

Indeed, Galewski and Stern proved in \cite[Theorem 2.1]{GS5} that (assuming the truth of Corollary~\ref{cor:no2}) an obstruction to the existence of simplicial triangulations on manifolds $M$ of dimension $\geq 5$ is the non-vanishing of $\Sq^1 \Delta(M) \in H^5(M; \Z/2)$, where $\Delta(M) \in H^4(M; \Z/2)$ is the Kirby-Siebenmann obstruction to combinatorial triangulations, and $\Sq^1$ denotes the first Steenrod square. It follows from the work of Galewski-Stern, Matumoto, and earlier work of Siebenmann \cite{SiebenmannQ} that all orientable $5$-manifolds are triangulable. A specific non-orientable five-dimensional manifold $M^5$ with $\Sq^1\Delta(M) \neq 0$ is constructed in \cite{GS5}. Hence $M^5$ is non-triangulable. To get non-orientable examples of non-triangulable manifolds in dimensions $n > 5$ we can take the product of $M^5$ with the torus $T^{n-5}$. To get an orientable example in dimension $6$ we can consider the non-orientable $S^1$-bundle over $M$ given by $\tilde M \times_{\Z/2} S^1$, where $\tilde M \to M$ is the oriented double cover; the total space of this bundle is orientable. To get orientable examples in dimensions $n > 6$, we can then take products with $T^{n-6}$.

Let us discuss the origin of the invariant $\beta$ from Theorem~\ref{thm:main}. There are two important antecedents. The first is Casson's invariant of integral homology spheres \cite{Casson}. Casson's invariant $\lambda$ satisfies the analogues of properties (1) and (2) in Theorem~\ref{thm:main} (anti-symmetry under orientation reversal, and being a lift of the Rokhlin invariant), but is not an invariant of homology cobordism. Nevertheless, this sufficed to make some progress in the direction of Corollary~\ref{cor:no2}:  A particular class of homology spheres $Y$ such that $Y \# Y$ is homology cobordant to $S^3$ is given by those $Y$ that admit an orientation reversing homeomorphism; using $\lambda$, one can show that homology spheres of this kind have Rokhlin invariant zero.

The second antecedent of $\beta$ consists of the ``correction terms'' in Floer homology inspired by the work of Fr{\o}yshov \cite{Froyshov}. Correction terms were defined first by Fr{\o}yshov in instanton (Yang-Mills) Floer homology \cite{FroyshovYM}, then by Ozsv\'ath-Szab\'o in Heegaard Floer homology \cite{AbsGraded}, and by Fr{\o}yshov and Kronheimer-Mrowka in monopole (Seiberg-Witten) Floer homology \cite{FroyshovHM, KMOS, KMbook}. In all these cases, one studies a version of Floer homology for $Y$, and captures a numerical invariant from its grading. 

For example, if $Y$ is a homology $3$-sphere, its monopole Floer homology $\HMto(Y)$ (as defined in \cite{KMbook}) is a graded module over the polynomial ring $\Z[U]$, with $U$ lowering degree by $2$. The module $\HMto(Y)$ consists of some $\Z[U]$-torsion part and a single ``infinite tail'' of the form:
\[ \xymatrixcolsep{.7pc}
\xymatrix{
 \Z  &  0 & \Z \ar@/_1pc/[ll]_{U} & 0 & \Z \ar@/_1pc/[ll]_{U} & 0 & \dots \ar@/_1pc/[ll]_{U} 
} \]
where the $U$-action is indicated by arrows. The Fr{\o}yshov invariant $h(Y)$ is defined as minus one-half of the minimal grading of an element in this tail. 

The correction terms mentioned above all satisfy analogues of the properties (1) and (3) in Theorem~\ref{thm:main} (anti-symmetry under orientation reversal, and a strong form of monotonicity under negative-definite cobordisms). However, none of them reduces to the Rokhlin invariant mod $2$.

The invariant $\beta$ combines the good properties of the Casson and Fr{\o}yshov-type invariants. It is defined as a correction term in a new, $\pin$-equivariant version of Seiberg-Witten Floer homology. This version uses an extra symmetry of the Seiberg-Witten equations that appears in the presence of a spin structure. The same symmetry was previously used with success in four dimensions, most notably in Furuta's proof of the $10/8$-Theorem \cite{Furuta}. 

In three dimensions, we use the extra symmetry to define $\pin$-equivariant Seiberg-Witten Floer homology with coefficients in the field $\F$ of two elements. Given a rational homology sphere $Y$ with a spin structure $\s$, its $\pin$-equivariant Seiberg-Witten Floer homology is a graded module over the ring  $\F[q,v]/(q^3),$ with $q$ and $v$ lowering degrees by $1$ and $4$, respectively. This Floer homology has an infinite tail of the form:
\[ \xymatrixcolsep{.7pc}
\xymatrix{
\dots & \F  &  \F \ar@/_1pc/[l]_{q} &  \F \ar@/_1pc/[l]_{q} & 0 & \F \ar@/^1pc/[llll]^{v} & \F \ar@/_1pc/[l]_{q} \ar@/^1pc/[llll]^{v} & \F \ar@/_1pc/[l]_{q} \ar@/^1pc/[llll]^{v} & 0 & \dots  \ar@/^1pc/[llll]^{v} & \dots \ar@/^1pc/[llll]^{v} & \dots \ar@/^1pc/[llll]^{v}
} \]
If we forget the action of $q$, the tail consists of three direct summands supported in three different degrees mod $4$. One defines three invariants $\alpha(Y, \s), \beta(Y, \s), \gamma(Y, \s)$ in terms of the the minimal possible gradings of elements in each of the three summands. The middle invariant $\beta$ is the one used in Theorem~\ref{thm:main}. The invariants $\alpha$ and $\gamma$ satisfy the exact analogues of properties (2) and (3) in Theorem~\ref{thm:main}, but they are less useful than $\beta$ because they get switched under orientation reversal: 
$$ \alpha(-Y, \s) = - \gamma(Y, \s), \ \ \ \gamma(-Y, \s) = - \alpha(Y, \s).$$ 
A key fact to be noted is that the tail in $\pin$-equivariant Seiberg-Witten Floer homology is periodic only mod $4$ (not mod $2$). This allows us to get a hold on the mod $2$ reductions of $\alpha$, $\beta$, and $\gamma$: The reductions all end up being equal to the Rokhlin invariant.

The construction of $\pin$-equivariant Seiberg-Witten Floer homology in this paper uses the techniques previously employed by the author in \cite{Spectrum}. It involves doing finite-dimensional approximations of the Seiberg-Witten equations, using Conley index theory to construct a $\pin$-equivariant space, and then taking the homology of this space. In \cite{Spectrum}, this was done in an $S^1$-equivariant context. Adapting the construction to the $\pin$-equivariant setting presents no major difficulties. 

We mention that we chose the methods in \cite{Spectrum} because (for rational homology $3$-spheres) they seemed easiest from a technical point of view. However, we expect that $\pin$-equivariant Seiberg-Witten Floer homology can also be defined in the spirit of the work of Kronheimer and Mrowka \cite{KMbook}. Furthermore, there should be a $\pin$-version of Heegaard Floer homology, where the role of the extra symmetry is played by the interchange of the alpha and beta curves. The advantage of using these theories (rather than the Conley index method) is that they should make it possible to define $\pin$-Floer homology for all spin $3$-manifolds, and to make it more computable.

\medskip
\noindent {\bf Acknowledgements.} I would like to thank Ron Fintushel, Mikio Furuta, Jin-Hong Kim, Peter Kronheimer, Timothy Nguyen, Nikolai Saveliev, Ron Stern, Raphael Zentner and the referees for helpful comments on  previous versions of this paper. 

\section{Spaces of type SWF}
\label{sec:spaces}

In this section we discuss some facts regarding the algebraic topology of spaces with a $\pin$-action. We show that under certain conditions, one can extract from their Borel homology groups three quantities $a, b, c \in \Z$. Later, in Section~\ref{sec:swf}, we will use this information in the context of Floer theory to obtain the three new invariants of homology cobordism.

\subsection{Pin(2)-equivariant topology}
 Let $\H = \C \oplus \C j= \{x + yi + zj + wk \mid x,y,z,w \in \R \}$ be the space of quaternions. Inside the group of unit quaternions $S(\H) = \su$ we have the circle group $S^1 = \C \cap S(\H)$, and also the subgroup
$$ G: = \pin = S^1 \cup S^1j.$$
There is a short exact sequence
\begin{equation}
\label{eq:pin}
 1 \To S^1 \To G \To \Z/2 \To 1.
 \end{equation}
 Furthermore, the inclusion $G\subset \su$ can be viewed as part of a fibration
 \begin{equation}
\label{eq:pinf}
  G \To \su \To \rp^2,
\end{equation}
where the second map is the composition of the Hopf fibration with the involution on $S^2$.

Among the real irreducible representations of $G$, we mention the following three:
\begin{itemize}
\item the trivial representation $\R$;
\item the one-dimensional sign representation $\tR$ on which $S^1 \subset G$ acts trivially and $j$ acts by multiplication by $-1$;
\item the quaternions $\H$, acted on by $G$ via left multiplication.
\end{itemize}

We want to study the topology of spaces with a $G$-action. Let us start by understanding the classifying space $BG = EG/G$. The short exact sequence \eqref{eq:pin} shows that $BG$ is the quotient of $BS^1 = \cp^{\infty}$ under the involution
$$ [z_1 : w_1 : z_2 : w_2 : \dots ] \ \to \ [-\bar{w}_1 : \bar{z}_1 : -\bar{w}_2 : \bar{z}_2 : \dots ].$$

An alternate (and more useful) description of $BG$ comes from \eqref{eq:pinf}, which gives a fibration
 \begin{equation}
\label{eq:fib2}
 \rp^2 \To BG \To B\su = \hp^{\infty}.
 \end{equation}

We are interested in the cohomology of $BG$ with coefficients in the field $\F = \Z/2$. The Leray-Serre spectral sequence associated to \eqref{eq:fib2} has no room for higher differentials, so the cohomology groups of $BG$ are
\[ \xymatrixcolsep{.7pc}
\xymatrix{
 \F  &  \F  &  \F  & 0 & \F  & \F  & \F  & 0 & \dots 
} \]
in degrees $0, 1, 2, \dots$ Moreover, the multiplicative properties of the spectral sequence show that, as a ring, 
\begin{equation}
\label{eq:bg}
\Ring := H^*(BG; \F) \cong \F[q, v]/ (q^3),
 \end{equation}
with elements $q$ in degree $1$ and $v$ in degree $4$. 

Let $X$ be a pointed, finite $G$-CW complex. Consider its (reduced) Borel homology and cohomology
$$ \tH_*^{G} (X; \F) = \tH_*(EG_+ \wedge_{G} X; \F),$$
$$ \tH^*_{G} (X; \F) = \tH^*(EG_+ \wedge_{G} X; \F),$$
where $EG_+$ denotes the union of $EG$ with a disjoint basepoint.

Both Borel homology and Borel cohomology are modules over $\tH^*_G(S^0; \F) = H^*(BG; \F) = \Ring$. Note that since we work with coefficients over a field, the Borel homology and Borel cohomology of $X$ (in any given grading) are dual vector spaces, and their $\Ring$-module structures are also related to each other by duality. For example, the description \eqref{eq:bg} of $\tH^*_G(S^0; \F)$ implies that the Borel homology of $S^0$ is:
\[ \xymatrixcolsep{.7pc}
\xymatrix{
 \F  &  \F \ar@/_1pc/[l]_{q} &  \F \ar@/_1pc/[l]_{q} & 0 & \F \ar@/^1pc/[llll]^{v} & \F \ar@/_1pc/[l]_{q} \ar@/^1pc/[llll]^{v} & \F \ar@/_1pc/[l]_{q} \ar@/^1pc/[llll]^{v} & 0 & \dots  \ar@/^1pc/[llll]^{v} & \dots \ar@/^1pc/[llll]^{v} & \dots \ar@/^1pc/[llll]^{v}
} \]
in degrees $0, 1, 2, \dots$, with the module structure indicated through the arrows.

One property of Borel cohomology that we need is a version of the localization theorem; see \cite[III (3.8)]{tomDieck} for a proof:

\begin{proposition}
\label{prop:loc}
 Suppose $A \subseteq X$ is a $G$-subcomplex such that the action of $G$ on $X - A$ is free. Then the inclusion of $A$ into $X$ induces an isomorphism on equivariant cohomology after inverting the element $v \in H^*(BG; \F)$; that is, we have an isomorphism of $\F[q,v,v^{-1}]/(q^3)$-modules: 
\begin{equation}
\label{eq:loc}
v^{-1} \tH^*_{G} (A; \F) \cong v^{-1} \tH^*_{G} (X; \F).
\end{equation}
\end{proposition}

Another important property of Borel cohomology (with $\F$ coefficients) is its invariance under suspensions, up to a shift in degree. Precisely, if $V$ is a finite-dimensional representation of $G$, let us denote by $V^+$ the one-point compactification of $V$, and by $\Sigma^V X = V^+ \wedge X$ the suspension of $X$ by this representation. We have: 

\begin{proposition}
\label{prop:susp}
 For any finite-dimensional representation $V$ of $G$, we have an isomorphism of $\Ring$-modules:
\begin{equation}
\label{eq:susp}
\tH^*_{G} (\Sigma^V X; \F) \cong \tH^{*- \dim V}_{G} (X; \F).
\end{equation} 
A similar isomorphism holds for Borel homology.
\end{proposition}

\begin{proof} There is a $V$-bundle:
 $$p:EG \times_{G} (V \times X) \to EG \times_{G} X.$$ 
Applying the relative Thom isomorphism theorem (with $\F$ coefficients) to this bundle over the pair $(EG \times_{G} X, EG \times \pt)$, we obtain \eqref{eq:susp}. 
\end{proof}

\begin{remark}
\label{rem:Z}
Borel homology and Borel cohomology with $\Z$ coefficients are not invariant under arbitrary suspensions. The analogue of the isomorphism \eqref{eq:susp} with $\Z$ coefficients holds if $V$ is a trivial representation, and also (using the relative Thom isomorphism theorem) if $V$ is a complex representation of $G$ such as $\H$. However, it does not hold for non-trivial real representations such as $\tR$. 
\end{remark}

\subsection{Equivariant duality}
\label{sec:dual}
The problem mentioned in Remark~\ref{rem:Z} can be fixed (with $\Z$ coefficients) by using the $RO(G)$-graded homology theory of Lewis, May and McClure \cite{LMM, LMS}. This theory is invariant under suspension by any representation, up to a corresponding shift in the $RO(G)$-grading. In this paper we will only make use of the $RO(G)$-graded theory indirectly: With coefficients in $\F$, we can collapse its grading to $\Z$ via the natural map $RO(G) \to \Z, \ V \mapsto \dim V$, and the result is the usual Borel homology. It follows that Borel homology (with $\F$ coefficients) satisfies the various properties established in the literature for $RO(G)$-graded homology.

In particular, we are interested in the behavior of Borel homology under equivariant Spanier-Whitehead duality. This was studied in \cite{Wirthmuller1, Wirthmuller2, LMS, GreenleesMay} in the context of $RO(G)$-graded homology. We gather below a few facts taken from these sources. We simplify the exposition so as to be in terms of Borel homology with $\F$ coefficients, and also in terms of spaces rather than spectra. 

Recall that in non-equivariant algebraic topology, two pointed, connected spaces $X$ and $X'$ are said to be Spanier-Whitehead $m$-dual if there is a map $\eps: X \wedge X' \to S^m$ such that slant product with the fundamental class of $S^m$ induces an isomorphism $H_*(X) \to H^{m-*}(X)$ in all degrees \cite{Spanier, SpanierWhitehead}. The equivariant analogue of this is $V$-duality, with respect to a representation $V$ of $G$. The original definition, as given in \cite[Definition 3.4]{LMS}, involves an isomorphism on the equivariant stable homotopy of the infinite suspensions of $X$ and $X'$. In the case of finite $G$-CW complexes, according to  \cite[Theorem 3.6]{LMS}, there is an equivalent and more elementary definition:

\begin{definition}
Let $V$ be a finite-dimensional representation of $G$. Two pointed, finite $G$-CW complexes $X$ and $X'$ are called {\em (equivariantly) $V$-dual} if there exists a $G$-map $\eps: X \wedge X' \to V^+$ such that for any subgroup $H \subseteq G$, the fixed point set map $\eps^H: X^H \wedge (X')^H \to (V^H)^+$ induces a non-equivariant duality between $X^H$ and $(X')^H$.
\end{definition}

The examples of $V$-dual spaces that we need in this paper come from the following:

\begin{lemma}
\label{lem:Vdual}
Suppose $N$ is a smooth $G$-manifold with boundary embedded in a representation $V$ of $G$, such that $\dim N = \dim V = m$. Further, suppose that $\del N$ admits a decomposition $\del N = L  \cup L'$ with $L$ and $L'$ being smooth $(m-1)$-dimensional $G$-manifolds such that $L \cap L' = \del L = \del L'$. Then $N/L$ and $N/L'$ are equivariantly $V$-dual.
\end{lemma} 

\begin{proof}
The non-equivariant analogue of this lemma is well-known; see for example \cite{AtiyahThom} or \cite[Lemma 3.6]{Cornea}. The equivariant version is a direct consequence of \cite[Theorem 4.1]{LMS}, which states that (under certain technical assumptions, automatically satisfied for embeddings of smooth $G$-manifolds), if we have inclusions of $G$-spaces $A \subset X \subset V$ as neighborhood retracts, then $X/A$ and $(V - A) / (V - X)$ are dual with respect to $V$. For the case at hand, take $(X, A) = (N, L)$ and observe that $(V - L) / (V-N)$ is $G$-equivalent to $N/L'$. 
\end{proof}

If $X$ and $X'$ are $V$-dual, the Borel homology of $X$ is related to the Borel cohomology of $X'$ as follows. Let $m = \dim V$. First, the Borel cohomology of $X'$ can be viewed as an equivariant generalized homology of $X$, called {\em co-Borel homology} and denoted by $cH_*^G$:
\begin{equation}
\label{eq:coBorel}
c\tH^G_*(X; \F) \cong \tH^{m-*}_G(X' ; \F) =  \tH^{m-*}(EG_+ \wedge_{G} X'; \F).
\end{equation}

Second, the Borel and co-Borel homologies fit into a long exact sequence, whose third term is another equivariant generalized homology $tH_*^G$, called {\em Tate homology}:
\begin{equation}
\label{eq:tate}
\dots \To c\tH_{*}^G(X; \F) \To t\tH_{*}^G(X; \F) \To \tH_{*-2}^G(X; \F) \To  \dots 
\end{equation}

Tate homology was defined by Greenlees and May \cite{GreenleesMay} as the co-Borel homology of $\tilde EG \wedge X$, where $\tilde EG$ is the unreduced suspension of $EG$ (with one of the cone points as basepoint). We will need the following facts:
\begin{itemize}

\item Being an (equivariant) generalized homology theory, Tate homology satisfies the usual Eilenberg-Steenrod axioms: homotopy, excision and exactness. 

\item Tate homology with coefficients in $\F$ is invariant under suspensions by arbitrary representations of $G$:
\begin{equation}
\label{eq:tateSusp}
t\tH_*^{G} (\Sigma^V X; \F) \cong t\tH_{*- \dim V}^{G} (X; \F).
\end{equation}
This holds because Tate homology is a particular case of co-Borel homology.

\item If $X$ has a free $G$-action away from the basepoint, then 
\begin{equation}
\label{eq:tate0}
t\tH_*^G(X; \F) = 0.
\end{equation}
This is proved in \cite[Proposition 2.4]{GreenleesMay}.

\item Tate homology is $4$-periodic, that is,
\begin{equation}
t\tH_*^G(X; \F) \cong t\tH_{*-4}^G(X; \F).
\end{equation}
Indeed, the right hand side is the Tate homology of $\Sigma^{\H}X$ by \eqref{eq:tateSusp}. On the other hand, $\Sigma^{\H} X$ contains $X$ as a subset whose complement has a free $G$-action, so \eqref{eq:tate0} and exactness imply that $t\tH_*^G(\Sigma^{\H}X; \F) \cong t\tH_*^G(X; \F)$.

\item For $X=S^0$ we have a graded isomorphism:
\begin{equation}
\label{eq:tateS}
t\tH_{*-2}^G(S^0; \F) \cong v^{-1}\Ring \cong \F[q, v , v^{-1}]/(q^3).
\end{equation}
Indeed, the Tate {\em co}homology $t\tH^*_G(S^0; \F)$ is isomorphic (as a graded vector space) to $\F[q, v , v^{-1}]/(q^3)$, according to the computation in \cite[Corollary 9.10]{GreenleesMay}. Further, we have $t\tH^*_G(S^0; \F) \cong t\tH_{-*}^G(S^0; \F)$ by \cite[p. 58]{GreenleesMay}. Of course, given the structure of $\F[q,v, v^{-1}]/(q^3)$, reversing its grading is the same as shifting it by $2$. From here we get \eqref{eq:tateS}  as an isomorphism of graded vector spaces. The fact that \eqref{eq:tateS} is also an isomorphism of $\Ring$-modules follows from the fact that the exact sequence \eqref{eq:tate} (applied to $S^0$) is one of modules, together with the $4$-periodicity of Tate homology as a module.
\end{itemize}

These properties makes Tate homology computable in many cases. Combining \eqref{eq:coBorel} with \eqref{eq:tate} and with knowledge of $t\tH_*$, one can get information about the Borel homology of $X$ in terms of the Borel cohomology of $X'$. 

\begin{remark}
The discussion above should be compared with its analogue in the $S^1$-equivariant case, which appeared in in \cite[Section 5.2]{PFP} in connection with $S^1$-equivariant Seiberg-Witten Floer homology. The readers familiar with Heegaard Floer homology should think of the Borel, co-Borel and Tate homologies as similar to $\HF^+, \HF^-$ and $\HF^{\infty}$; compare \cite[Conjecture 1]{PFP}.
\end{remark}

\subsection{Spaces of type SWF} \label{sec:spacesSWF} The definition below is motivated by the construction of the Seiberg-Witten Floer spectrum in Section~\ref{sec:swf}. We let $G= \pin$ as before.

\begin{definition}
Let $s \geq 0$. A {\em space of type $\swf$ (at level $s$)} is a pointed, finite $G$-CW complex $X$ with the following properties:
\begin{enumerate}[(a)]
\item The $S^1$-fixed point set $X^{S^1}$ is $G$-homotopy equivalent to the sphere $(\tR^s)^+$;
\item The action of $G$ is free on the complement $X - X^{S^1}$.
\end{enumerate}
\end{definition}

Given a space $X$ of type $\swf$ at level $s$, we can apply Proposition~\ref{prop:loc} to the subcomplex $A = X^{S^1}$ and obtain
\begin{equation}
\label{eq:locswf}
v^{-1} \tH^*_G(X; \F) \cong v^{-1} \tH^*_G( (\tR^s)^+; \F) \cong (v^{-1}\Ring)_{[s]},
\end{equation}
where the last isomorphism follows from \eqref{eq:bg} and the suspension invariance of Borel cohomology (Proposition~\ref{prop:susp}). The notation $[s]$ indicates a grading shift by $s$, so that the element $1 \in \F[q, v, v^{-1}]/(q^3) \cong v^{-1}\Ring$ is moved to degree $s$.

Equation~\eqref{eq:locswf} implies that there are nonzero elements in $\tH^*_G(X; \F)$ in some degrees congruent to $s, s+1$ and $s+2  \ (\mod 4)$, such that these elements do not get killed by inverting $v$; that is, for any $l \geq 0$, multiplying them by $v^l$ does not give zero.

Thus, to the space $X$ we can associate the following three quantities:
\begin{align*}
a(X) &= \min \{  r \equiv s \ (\mod 4) \mid \exists \ x \in \tH^r_G(X; \F), \ v^l x \neq 0 \text{ for all } l \geq 0  \},\\
b(X) &= \min \{  r \equiv s+1 \ (\mod 4) \mid  \exists \ x \in \tH^r_G(X; \F), \ v^l x \neq 0 \text{ for all } l \geq 0 \} - 1,\\
c(X) &= \min \{ r \equiv s+2 \ (\mod 4) \mid  \exists \ x \in \tH^r_G(X; \F), \ v^l x \neq 0 \text{ for all } l \geq 0 \} - 2.
\end{align*}

Concretely, the Borel cohomology of $X$ looks like the Borel cohomology of $S^0$ in high enough degrees (after a grading shift by $s$). Indeed, forgetting the action of $q$ for the moment, we see that as an $\F[v]$-module, $\tH^*_G(X; \F)$ consists of some $\F[v]$-torsion part and three summands isomorphic to $\F[v]$, supported in degrees congruent to $s$, $s+1$ and $s+2$ modulo $4$. Since $X$ is a finite CW complex, we have that $\tH^*_G(X; \F)$ is finitely generated as an $\F[v]$-module, so its  $\F[v]$-torsion part is bounded above in grading. The quantities $a(X)$, $b(X)+1$ and $c(X) + 2$ describe the grading of $1 \in \F[v]$ in each of the $\F[v]$-free summands. 

It is clear from the construction that
\begin{equation}
\label{eq:mod4}
 a(X) \equiv b(X) \equiv c(X) \equiv s \pmod{4}.
\end{equation}
and that $a(X), b(X), c(X) \geq 0$.

To explore the properties of $a, b, c$ further, it is helpful to introduce an ``infinity'' version of Borel cohomology: 
$$ \iH^*_G(X; \F) = \text{image }\bigl (\tH^*_G(X; \F) \To v^{-1}\tH^*_G(X; \F) \bigr).$$

Observe that $\iH^*_G(X; \F)$ can be identified with the quotient of $\tH^*_G(X; \F)$ by the kernel of $v^l$ for $l \gg 0$. This quotient $\Ring$-module is the $\F[v]$-free part of $\tH^*_G(X; \F)$, consisting of the three summands mentioned above. 

Note also that $\iH^*_G(X; \F)$ is a graded $\Ring$-submodule of $v^{-1}\tH^*_G(X; \F) \cong (v^{-1}\Ring)_{[s]}$ supported in non-negative degrees. Let $n$ be a negative number congruent to $s$ modulo $4$. The submodule of $(v^{-1}\Ring)_{[s]}$ consisting of all elements in grading $\geq n$ can be identified with $\Ring_{[n]}$, so that $\iH^{*+n}_G(X; \F)$ becomes a graded ideal of the ring $\Ring$. Moreover, we must have $v^{-1} \bigl( \iH^{*+n}_G(X; \F)\bigr) = v^{-1}\Ring$.

The following lemma describes all the possibilities for the ideal $\iH^{*+n}_G(X; \F)$:
\begin{lemma}
\label{lem:yy}
Let $\J$ be a graded ideal of $\Ring = \F[q, v]/(q^3)$ such that $v^{-1}\J = v^{-1}\Ring$. Then
$$ \J = (v^i, qv^j, q^2 v^k),$$
for some $i \geq j \geq k \geq 0$.
\end{lemma}
\begin{proof}
Let us ignore the action of $q$ for the moment. Thus, we view $\Ring$ as a graded $\F[v]$-module and $\J$ as a graded submodule of $\Ring$. We have a decomposition $\J =  \J_0 \oplus \J_1 \oplus \J_2$, where $\J_s$ denotes the part of $\J$ supported in degrees $\equiv s \ (\mod 4)$. Then $\J_0$ is a graded submodule of $\F[v]$, and it is nontrivial because $v^{-1}\J = v^{-1}\Ring$ implies $v^{-1} \J_0 = \F[v, v^{-1}]$. Therefore, we must have $\J_0 = (v^i)$ for some $i \geq 0$. Similarly, we see that $\J_1 = (qv^j)$ and $\J_2 = (q^2 v^k)$ for some $j, k \geq 0$. 

We now consider the action of $q$. Since $\J$ must be invariant under this action and $v^i \in \J_0 \subset \J$, we deduce that $qv^i \in \J_1$, so $i \geq j$. Similarly, we get $j \geq k$.
\end{proof}

If $\iH^{*+n}_G(X; \F) = (v^i, qv^j, q^2 v^k)$, we see that
$$ a(X) = 4i+n, \ b(X)= 4j + n, \ c(X) = 4k+n.$$

It follows from Lemma~\ref{lem:yy} that:
\begin{equation}
\label{eq:ineq}
 a(X) \geq b(X) \geq c(X). 
\end{equation}

Observe that if $X$ is a space of type $\swf$ at level $s$, then the suspensions $\Sigma^{\tR} X$ and $\Sigma^{\H} X$ are of type $\swf$ at levels $s+1$ and $s$, respectively.

\begin{lemma}
\label{lem:abcSusp}
Let $X$ be a space of type $\swf$, and $V$ a representation of $G$ of the form $\tR^n \oplus \H^p$, for some $n, p \geq 0$. Then:
$$ a(\Sigma^V X) = a(X) + \dim V, \ b(\Sigma^V X) = b(X) + \dim V, \ c(\Sigma^V X) = c(X) + \dim V.$$
\end{lemma}

\begin{proof}
This follows immediately from Proposition~\ref{prop:susp}.
\end{proof}

Finally, let us note that we can alternatively describe $a(X), b(X), c(X)$ in terms of Borel homology. If $x$ is a nonzero element of $\tH^r_G(X; \F)$, we can complete $x$ to a basis of Borel cohomology in degree $r$ and construct a dual element $x^*$ in Borel homology. The condition $v^l x \neq 0$ is equivalent to $0 \neq x^* \in \im(v^l)$. Therefore, if we let
\begin{equation}
\label{eq:iH}
\iH_*^G(X; \F) := \bigcap_{l \geq 0} \im \bigl (v^l : \tH_{*+4l}^G(X; \F) \To \tH_*^G(X; \F) \bigr),
\end{equation}
we obtain:
\begin{align}
a(X) &= \min \{  r \equiv s \ (\mod 4) \mid \exists \ x,\ 0 \neq x \in  {\iH}_r^G(X; \F)  \}, \label{eq:ax}\\
b(X) &= \min \{  r \equiv s +1 \ (\mod 4) \mid \exists \ x,\ 0 \neq x \in {\iH}_r^G(X; \F) \} - 1, \label{eq:bx}\\
c(X) &= \min \{  r \equiv s+2 \ (\mod 4) \mid \exists \ x, \ 0 \neq x \in {\iH}_r^G(X; \F)  \} - 2 \label{eq:cx}.
\end{align}

 \subsection{Examples}
 The simplest example of a space of type $\swf$ is $S^0$, for which we have
 $$ a(S^0) = b(S^0)= c(S^0) = 0.$$ 
 
To obtain more interesting examples, suppose that $G$ acts freely on a finite $G$-CW-complex $Z$, and let $Q = Z/G$ be the respective quotient. Let $$\tilde Z = \bigl( [0,1]\times Z \bigr ) / (0,z) \sim (0, z') \text{ and } (1,z) \sim (1, z') \text{ for all } z,z' \in Z$$ denote the {\it unreduced} suspension of $Z$, where $G$ acts trivially on the $[0,1]$ factor. We view $\tilde Z$ as a pointed $G$-space, with one of the two cone points being the basepoint. Clearly $\tilde Z$ is of type $\swf$, with $(\tilde Z)^{S^1} = S^0$. Note that the cone of the inclusion of $(\tilde Z)^{S^1}$ into $\tilde Z$ is the reduced suspension $\Sigma^{\R} Z_+$. Information about the Borel cohomology of $\tilde Z$ can be extracted from the long exact sequence:
$$ \dots \To \tH^*_G(\tilde Z; \F) \To \tH^*_G(S^0; \F) \To \tH^{*+1}_G(\Sigma^{\R}Z_+; \F) \To \dots$$
Since $G$ acts freely on $Z$, we have $\tH^{*+1}_G(\Sigma^{\R}Z_+; \F) \cong \tH^*(Z_+) \cong H^*(Q)$, so the above sequence can be written
\begin{equation}
\label{eq:z}
\dots \To \tH^*_G(\tilde Z; \F) \To H^*(BG; \F) \xrightarrow{\phantom{b}\kappa^*} H^*(Q; \F) \To \dots 
\end{equation}
The map $\kappa^*$ is induced from the map $\kappa: Q \to BG$ that classifies the $G$-bundle $Z$ over $Q$. The image of $\kappa^*$ produces the $G$-characteristic classes of that bundle.

\begin{example}
\label{ex:G}
Let $Z = G$, acting on itself via left multiplication, so that the quotient $Q$ is a single point. As a topological space, $\tilde G$ is the suspension of two disjoint circles. In the exact sequence \eqref{eq:z}, the map $\kappa^*$ is an isomorphism in degree $0$. We deduce that $\tH^*_G(\tilde G; \F)$ is isomorphic to the submodule of $H^*(BG; \F) \cong \F[q,v]/(q^3)$ consisting of the elements in degrees $\geq 1$. Therefore,
$$ a(\tilde G) = 4, \ \ b(\tilde G) = c(\tilde G) = 0.$$ 
\end{example}

\begin{example}
\label{ex:zn}
More generally, for $n \geq 1$, let 
$$Z_n = \bigl( S(\C^n) \times \{0\} \bigr) \cup \bigl( \{ 0 \} \times S(j\C^n) \bigr) \subset \C^n \oplus j\C^n \cong \H^n.$$
This is a $G$-invariant subset of $\H^n$, with quotient $Q_n = \cp^{n-1}$. (In particular, $Z_1 = G$ is the previous example.) The bundle $Z_n \to \cp^{n-1}$ can be viewed as induced from the $S^1$-bundle $S(\C^n) \to \cp^{n-1}$ via the monomorphism $S^1 \to G$. Thus, the classifying map $\kappa: \cp^{n-1} \to BG$ factors as
$$ \cp^{n-1} \lhook\joinrel\To \cp^{\infty} \cong BS^1 \xrightarrow{\phantom{b}\eta} BG.$$

We can figure out the map induced by $\eta$ on cohomology by noticing that there is a fiber bundle 
$$S^2 \To \cp^{\infty} \To \hp^{\infty}$$
which double covers the bundle \eqref{eq:fib2}. Using the functorial properties of the Leray-Serre spectral sequences, we see that $\eta^* : H^*(BG; \F) \to H^*(\cp^{\infty}; \F)$ is an isomorphism in degrees divisible by $4$, and zero otherwise. It follows that the map 
$$\kappa_*: H^*(BG; \F) \to H^*(\cp^{n-1}; \F)$$ is an epimorphism in degrees divisible by $4$, and zero otherwise. Using this information, we deduce from the exact sequence~\eqref{eq:z} that:
\begin{equation}
\label{eq:azn}
 a(\tilde Z_n) = 4\lceil n/2 \rceil, \ \ b(\tilde Z_n) = c(\tilde Z_n) = 0.
 \end{equation}
\end {example}

\begin{example}
\label{ex:znp}
For $n \geq 1$, let
$$ Z'_n = S(\C^n) \times S(j\C^n) \subset \C^n \oplus j\C^n \cong \H^n.$$
This is a $G$-invariant subset of $\H^n$. We shall see in the next subsection (Example~\ref{ex:dual}) that:
\begin{equation}
\label{eq:zpn}
 a(\tilde Z'_n) =  b(\tilde Z'_n) = 4n, \ \ c(\tilde Z'_n) = 4n - 4\lceil n/2 \rceil.
 \end{equation}
\end{example}

\subsection{Other properties}
Let us understand the behavior of the invariants $a, b, c$ under equivariant Spanier-Whitehead duality:

\begin{proposition}
\label{prop:abcDuality}
Suppose $X$ and $X'$ are spaces of type $\swf$ that are $V$-dual, for some $G$-representation $V \cong \tR^n \oplus \H^p$. Then:
\begin{equation}
\label{eq:abcxx}
 a(X') = \dim V -c(X), \ \ b(X') = \dim V - b(X), \ \ c(X') = \dim V - a(X).
 \end{equation}
\end{proposition}

\begin{proof}
Let $m = \dim V = n + 4p$, and assume that $X$ is of type $\swf$ at level $s$. By the definition of $V$-duality, the fixed point sets $X^{S^1} \cong (\tR^s)^+$ and $(X')^{S^1}$ are non-equivariantly $n$-dual. Therefore, $X'$ must be of type $\swf$ at level $n-s$.

Using the isomorphism \eqref{eq:coBorel} and the exact sequence \eqref{eq:tate}, we see that the Borel homology of $X$ is related to the Borel cohomology of $X'$ by a long exact sequence of $\Ring$-modules:
\begin{equation}
\label{eq:abcDual}
\dots \To \tH^{m-*}_G(X'; \F) \To t\tH_{*}^G(X; \F) \To \tH_{*-2}^G(X; \F) \To  \dots 
\end{equation}

The Tate homology of $X$ can be computed using the properties mentioned at the end of Subsection~\ref{sec:dual}. Since $X - X^{S^1}$ has a free $G$-action, Equations~\eqref{eq:tate0} and \eqref{eq:tateS} combined with excision and suspension invariance imply that
$$ t\tH_*^G(X; \F) \cong t\tH_*^G(X^{S^1}; \F) \cong t\tH_{*-s}^G(S^0; \F) \cong (v^{-1}\Ring)_{[s+2]}.$$

Let us study the exact sequence \eqref{eq:abcDual} more closely. Since every element of $v^{-1}\Ring$ is in the image of $v^l$ for all $l$, it follows that the map from $ t\tH_{*}^G(X; \F)$ to $\tH_{*-2}^G(X; \F)$ factors through the submodule $\iH_{*-2}^G(X; \F)$. Similarly, since $v^l$ is an isomorphism on $v^{-1}\Ring$ for all $l$, the map from $\tH^{m-*}_G(X'; \F)$ to $t\tH_{*}^G(X; \F)$ must takes the kernel of $v^l$ to zero, i.e., it must factor through the quotient module $\iH^{m-*}_G(X'; \F)$. (Recall that $\iH^{*}_G$ was identified with the quotient of $H^{*}_G$ by $v^l$ for $l \gg 0$.) Therefore, we can write a trimmed version of the exact sequence \eqref{eq:abcDual} that involves only the ``infinity'' parts of Borel homology and cohomology. After shifting degrees by $s+2$, it reads:
\begin{equation}
\label{eq:abcDual2}
\dots \To {\iH}^{m-s-2-*}_G(X'; \F) \To  (v^{-1}\Ring)_* \To {\iH}_{*+s}^G(X; \F) \To  \dots 
\end{equation}

The exactness of \eqref{eq:abcDual2} follows from the exactness of \eqref{eq:abcDual}.

Recall that the infinity parts of Borel homology and cohomology are determined by the invariants $a$, $b$ and $c$. Precisely, we have:
$$ {\iH}^{m-s-2-r}_G(X'; \F) = \begin{cases}
\F & \text{ if } r = m-a(X')-s-4j-2, \ j \geq 0,\\
\F & \text{ if } r = m-b(X')-s-4j-3, \ j \geq 0,\\
\F & \text{ if } r = m-c(X')-s-4j-4, \ j \geq 0,\\
0 & \text{ otherwise,}
\end{cases}$$
and
$$
{\iH}_{r+s}^G(X; \F) =  \begin{cases}
\F & \text{ if } r = a(X)-s+4j, \ j \geq 0,\\
\F & \text{ if } r = b(X)-s+4j+1, \ j \geq 0,\\
\F & \text{ if } r = c(X)-s+4j+2, \ j \geq 0,\\
0 & \text{ otherwise.}
\end{cases}$$

Note that $a(X), b(X)$ and $c(X)$ are all congruent to $s$ modulo $4$, and similarly $a(X'), b(X')$ and $c(X')$ are all congruent to $n-s$ (hence to $m-s=n+4p-s$) modulo $4$. Taking into account the inequalities \eqref{eq:ineq} for the invariants of $X$ and $X'$, an analysis of the exact sequence \eqref{eq:abcDual2} shows that we must have the desired constraints \eqref{eq:abcxx}.
\end{proof}

\begin{example}
\label{ex:dual}
Let $Z_n$ be the space considered in Example~\ref{ex:zn}. Its unreduced suspension $\tilde Z_n$ can be identified with the subset
$$   \bigl( \C^n \times \{0\} \bigr) \cup \bigl( \{ 0 \} \times j\C^n \bigr) \cup \{\infty\} \subset (\C^n \oplus j\C^n)^+ \cong (\H^n)^+  $$
The quotient $(\H^n)^+/\tilde Z_n$ admits a $G$-equivariant deformation retraction onto the unreduced suspension $\tilde{Z_n'}$ of the space $ Z_n' = S(\C^n) \times S(j\C^n)$ from Example~\ref{ex:znp}. It follows from  \cite[Theorem 4.1]{LMS} that $\tilde{Z_n'}$ is $\H^n$-dual to $\tilde{Z_n}$. The calculations in \eqref{eq:azn} together with Proposition~\ref{prop:abcDuality} imply the results for $a(\tilde Z_n'), b(\tilde Z_n'), c(\tilde Z_n')$ stated in \eqref{eq:zpn}.
\end{example}

Another useful result is the behavior of $a(X), b(X), c(X)$ under a certain kind of equivariant maps:
\begin{proposition}
\label{prop:abcIneq}
Suppose $X$ and $X'$ are spaces of type $\swf$ at the same level $m$, and suppose that $f: X \to X'$ is a $G$-equivariant map whose $S^1$-fixed point set map is a $G$-homotopy equivalence. Then:
$$ a(X) \leq a(X'), \ b(X) \leq b(X'), \ c(X) \leq c(X').$$
\end{proposition}

\begin{proof}
Since the $S^1$-fixed point set map associated to $f$ is a $G$-homotopy equivalence, the functoriality of the localization maps in \eqref{eq:locswf} implies that $f$ induces an isomorphism on Borel cohomology after inverting $v$. Given the structure of Borel cohomology for a space of type $\swf$, this means that $f$ induces an isomorphism on Borel cohomology in large enough degrees. By taking duals, we see that the map induced by $f$ on Borel homology:
$$ f_*: \tH_*^G(X; \F) \to \tH_*^G(X'; \F)$$
must also be an isomorphism in large enough degrees. Since $f_*$ commutes with the action of $v$, it must map the submodule $\iH_*^G(X; \F)$ to $\iH_*^G(X'; \F)$. 

Suppose we have a nonzero element $x' \in {\iH}_r^G(X'; \F)$ in some grading $r$. For any $l$ we can find some $y' \in {\iH}_{r+4l}^G(X'; \F)$ with  $x' = v^l y'$. If we choose $l$ large enough, $y'$ must be of the form $f_*(y)$ for some $y \in {\iH}_{r+4l}^G(X; \F)$. Let $x = v^l y \in \tH_r^G(X; \F).$ Since $y$ is in $ {\iH}_{*}^G(X; \F)$, so is $x$; moreover, we have $f_*(x) = x'$, so $x$ is nonzero. Thus, we found a nonzero element in $\iH_r^G(X; \F)$. Using the definitions of $a, b, c$ in terms of Borel homology, the desired inequalities follow. 
\end{proof}

\subsection{A fourth quantity} \label{sec:fourth} Let us describe another numerical invariant, $d_p(X)$, that can be associated to a space $X$ of type $\swf$. The quantity $d_p(X)$ depends  also on the choice of $p \in \Z$, which can be either zero or a prime. Let $\f$ be a field of characteristic $p$. Instead of the $\pin$-equivariant cohomology of $X$, we use the $S^1$-equivariant cohomology with coefficients in $\f$:
$$ \tH^*_{S^1}(X; \f) = \tH^*(ES^1_+ \wedge X ; \f),$$
which is a module over $H^*_{S^1}(\pt; \f) = \f[U]$, with $U$ in degree $2$. Since $S^1 \subset G$ acts trivially on $\tR$, suspension by $\R$ preserves $S^1$-equivariant cohomology, up to a degree shift. The same is true for suspension by the complex representation $\H$; see Remark~\ref{rem:Z}. This is why in the $S^1$-equivariant case we can let the field $\f$ be arbitrary.

Localization with respect to $S^1$ shows that, if $X$ is of type $\swf$ at level $s$, 
$$ U^{-1} \tH^*_{S^1}(X; \f) \cong U^{-1} \tH^*_{S^1}((\tR^s)^+; \f) \cong \f[U, U^{-1}]_{[s]}$$
We define
$$  \iH^*_{S^1}(X; \f) = \text{image }\bigl (\tH^*_{S^1}(X; \f) \To U^{-1}\tH^*_{S^1}(X; \f) \bigr).$$
and let $d_p(X)$ be the minimal degree of a non-zero element in $\iH^*_{S^1}(X; \f)$. The same methods as in the $G$-equivariant case can be used to prove the analogues of Lemma~\ref{lem:abcSusp}, Proposition~\ref{prop:abcDuality} and Proposition~\ref{prop:abcIneq} for $d_p$ instead of $a, b, c$. In particular, if $X$ and $X'$ are $V$-dual, we have
$$ d_p(X') = \dim V - d_p(X).$$

The $S^1$-equivariant and $G$-equivariant cohomologies of a space are related to each other by equivariant transfer;  see for example \cite[Ch. III]{Bredon} or \cite[Proposition 9.13]{tomDieck}. Precisely, the element $j \in G$ induces an involution on $S^1$-equivariant cohomology. Let $\bigl ( H^*_{S^1} (X; \f) \bigr)^j$ denote the fixed point set of this involution. If $p \neq 2$, we have a transfer isomorphism:
$$ H^*_{G} (X; \f) \cong \bigl ( H^*_{S^1} (X; \f) \bigr)^j.$$ 
However, equivariant transfer fails over $\F$, and the invariants $a, b, c$ are defined in terms of Borel cohomology with coefficients in $\F$. Because of this, there does not seem to be any direct relation between $d_p$ and the invariants $a, b, c$. 

\begin{example}
Consider the space $Z_n$ from Example~\ref{ex:zn}, with unreduced suspension $\tilde Z_n$. The same methods used to compute $a(Z_n)$, $b(Z_n)$ and $c(Z_n)$ can be applied to show that 
 $$d_p(\tilde Z_n) = 2n,$$ for any $p$. In particular, $\tilde Z_0$ and $\tilde Z_1$ have the same values of $a$, $b$ and $c$, but different values for $d_p$. 
\end{example}

\section{Pin(2)-equivariant Seiberg-Witten Floer homology}
\label{sec:swf}

In this section we review the construction of the Seiberg-Witten Floer spectrum from \cite{Spectrum}, and explain the changes needed to take into account the full $\pin$-symmetry of the equations. Our focus is not the spectrum itself, but rather its equivariant homology. This homology can be defined without reference to spectra and, to keep the exposition simple, this is what we do. For completeness, we also discuss the construction of the Floer spectrum, but only briefly---in Subsection~\ref{sec:spectrum}.

Here is a rough sketch of what follows: Given a rational homology $3$-sphere $Y$, we consider a finite dimensional approximation of the Seiberg-Witten equations on $Y$. (This is inspired by a similar approximation used by Bauer and Furuta for the Seiberg-Witten equations in four dimensions \cite{BauerFuruta, Furuta}.) In three dimensions, the approximation takes the form of a gradient flow, to which one can associate a based $\pin$-space called the (equivariant) Conley index. The $\pin$-homotopy type of the Conley index is invariant under deformations. As we consider larger finite-dimensional approximations, the Conley index changes by suspension. We take its suitably normalized Borel homology to be our definition of the $\pin$-equivariant Seiberg-Witten Floer homology. We show that the Conley index is a space of type $\swf$, so (using the methods from  Section~\ref{sec:spaces}) we can extract from it three quantities $\alpha$, $\beta$, and $\gamma$. We then show that these invariants satisfy the properties advertised in the introduction, and compute them in a few examples.

\subsection{The Seiberg-Witten equations}
Let $Y$ be a rational homology three-sphere, $g$ a metric on $Y$, $\s$ a spin structure on $Y$, and $\Spin$ the spinor bundle for $\s$. Let $\rho:TY \to \text{End}(\Spin)$ denote the Clifford multiplication, and $\dirac : \Gamma(\Spin) \to \Gamma(\Spin)$ the Dirac operator. Consider the configuration space: 
\[
\Conf(Y,\s) = i\Omega^1(Y) \oplus \Gamma(\Spin).
\]

The gauge group $\G = C^\infty(Y,S^1)$ acts on $\Conf(Y,\s)$ by $u \cdot (a,\phi) = (a - u^{-1}du,u \cdot \phi)$. We define the normalized gauge group $\G_0$ to consist of those $u=e^{i\xi} \in \G$ such that $\int_Y \xi = 0$. We have a global Coulomb slice: 
$$V = i \ker d^* \oplus \Gamma(\Spin) \subset \Conf(Y, \s).$$ 
Given $(a,\phi) \in \Conf(Y,\s)$, there is a unique element of $V$ which is obtained from $(a,\phi)$ by a normalized gauge transformation; we call this element the {\em Coulomb projection} of $(a, \phi)$.

In addition to the gauge symmetry, we have an action of the group $G = \pin = S^1 \cup S^1j$ on $\Conf(Y, \s)$. Indeed, since the spin bundle $\Spin$ has structure group $\su = S(\H)$, there is a natural action of $G$ on $\Gamma(\Spin)$ by left multiplication.\footnote{If we identify the spinor bundle with $\C^2$, then the $j$ action takes $(v, w)$ to $(-\bar w, -\bar v)$. We further identify $\C^2$ with the quaternions by $(v, w) \to v + wj$, and then $j$ acts by left multiplication. Our conventions are different from \cite{Furuta}, where $G$ acted on the right.}  On forms $a \in i\Omega^1(Y)$, we let $S^1 \subset G$ act  trivially, and $j \in G$ act by multiplication by $-1$. Note that the action of $S^1 \subset G$ on $\Conf(Y, \s)$ coincides with that of the constant gauge transformations. Observe also that the Coulomb slice $V$ is preserved by the $G$-action.

Next, consider the Chern-Simons-Dirac functional $\csd: \Conf(Y, \s) \to \R$, given by:
\[
\csd(a,\phi) = \frac{1}{2} (\int_Y \langle \phi, \dirac \phi + \rho(a)\phi \rangle dvol - \int_Y a \wedge da).  
\]
Its critical points are the solutions to the Seiberg-Witten equations:
\[
*da + \tau(\phi,\phi) = 0, \; \; \dirac \phi + \rho(a)\phi = 0,  
\]
where $\tau(\phi,\phi) = \rho^{-1}(\phi \otimes \phi^*)_0 \in \Omega^1(Y; i\R)$. One can readily check that the $\csd$ functional is gauge invariant and $G$-invariant. 

Consider the restriction of $\csd$ to the global Coulomb slice $V$. By measuring the length of the projections of tangent vectors to local Coulomb slices, we obtain a Riemannian metric $\tilde g$ on $V$   such that the trajectories of the gradient flow of $\csd|_{V}$ are the Coulomb projections of the original gradient flow trajectories in $\Conf(Y,\s)$.  In $V$ with the metric $\tilde{g}$, we can write the flow trajectories of $\nabla (\csd|_{V})$ as 
\[
\frac{\partial}{\partial t} x(t) = - (\ell + c)(x(t)), 
\]
for 
\begin{eqnarray*}
\ell(a,\phi) &=& (*da, \dirac \phi) \\
c(a,\phi) &=& (\pi \circ \tau(\phi,\phi),\rho(a)\phi - i \xi(\phi)\phi),
\end{eqnarray*}
where $\pi:\Omega^1(Y; i\R) \to i\ker d^*$ denotes the orthogonal projection, and $\xi(\phi) \in \Omega^0(Y)$ is determined by $d\xi(\phi) = i(1-\pi) \circ \tau(\phi, \phi)$ and $\int_Y \xi(\phi)=0$.

For any integer $k \geq 0$, let $V_{(k)}$ denote the completion of $V$ with respect to the $L^2_k$ Sobolev norm. The gradient of $\csd$ on $V$ extends to a map
$$ \ell + c: V_{(k+1)} \to V_{(k)},$$
such that $\ell$ is a linear Fredholm operator, and $c$ is compact. The corresponding flow lines are called {\em Seiberg-Witten trajectories} (in Coulomb gauge). A Seiberg-Witten trajectory $x=(a, \phi): \R \to V$ is said to be {\em of finite type} if $\csd(x(t))$ and $\|\phi(t)\|_{C^0}$ are bounded in $t$. 

Note that both $\ell$ and $c$ are $G$-equivariant maps.

\subsection{Finite-dimensional approximation and the Conley index}
Let $\vnt$ be the finite-dimensional subspace of $V$ spanned by the eigenvectors of $\ell$ with eigenvalues in the interval $(\tau,\nu]$.  The orthogonal projection from $V$ to $\vnt$ will be denoted $\tilde{p}^\nu_\tau$. We want to modify it to make it smooth in $\nu$ and $\tau$. To do this, choose a smooth, non-negative function $\chi: \R \to \R$ that is non-zero exactly on $(0,1)$, and such that $\int_\R \chi(\theta) d\theta =1$. Then set
\[
p^\nu_\tau = \int^1_0 \chi(\theta) \tilde{p}^{\nu - \theta}_{\tau + \theta} d \theta.
\]

Consider the restriction of $\csd$ to $\vnt$. The gradient flow equation becomes:
\[
(\ell + \pml c)(x(t)) = - \frac{\partial}{\partial t} x(t). 
\]
We refer to its solutions as  {\em approximate Seiberg-Witten trajectories}. 

Fix a natural number $k \geq 4$. There exists a constant $R > 0$, such that all Seiberg-Witten trajectories $x=(a, \phi): \R \to V$ of finite type are contained in $B(R)$, the ball of radius $R$ in $V_{(k+1)}$.  The following is a corresponding  compactness result for approximate Seiberg-Witten trajectories:

\begin{proposition}[Proposition 3 in \cite{Spectrum}] \label{prop:3}
For any $\nu$ and $-\tau$ sufficiently large (compared to $R$), if $x:\mathbb{R} \to \vnt$ is a trajectory of the gradient flow $\ell + \pml c$, and $x(t)$ is in $\overline{B(2R)}$ for all $t$, then in fact $x(t)$ is contained in $B(R)$.
\end{proposition}

Pick a smooth, $G$-equivariant function $u: \vnt \to \R$ that vanishes outside $B(3R)$ and is the identity on $\overline{B(2R)}$. Then $u (\ell + \pml c)$ is a compactly supported vector field on $\vnt$, which generates a global flow on $\vnt$:
$$\varphi^\nu_\tau = \{(\varphi^\nu_\tau)_t: \vnt \to \vnt \}_{t \in \R}.$$ 

Let us recall a few basic notions of Conley index theory, following \cite{ConleyBook}. If we have a one-parameter subgroup $\varphi= \{\varphi_t \}$ of diffeomorphisms of a manifold $M$, and a compact subset $A \subseteq X$, define 
\[
\Inv (A, \phi) = \{x \in A \mid \varphi_t (x) \in A \text{ for all } t \in \mathbb{R} \}.
\]

We say that a compact subset $S \subseteq M$ is an {\em isolated invariant set} if it has an isolating neighborhood $A$, that is, a compact set $A \subseteq M$ such that $S = \Inv(A, \varphi) \subseteq \text{int}(A)$.  

\begin{definition}
\label{def:indexpair}
Let $S$ be  an isolated invariant set. An {\em index pair} $(N,L)$ for $S$ is a pair of compact sets $L \subseteq N \subseteq M$ such that:
\begin{enumerate}[(i)]
\item $\Inv (N - L, \varphi) = S \subset \text{int }(N - L)$, 
\item For all $x \in N$, if there exists $t >0$ such that $\varphi_t(x)$ is not in $N$, there exists $0 \leq \tau < t$ with $\varphi_\tau(x) \in L$ $(L$ is an {\em exit set} for $N)$, 
\item Given $x \in L$, $t>0$, if $\varphi_s(x) \in N$ for all $0 \leq s \leq t$, then $\varphi_s(x)$ is in $L$ for $0 \leq s \leq t$ $(L$ is {\em positively invariant in $N$)}. 
\end{enumerate}
\end{definition}

It was proved by Conley \cite{ConleyBook} that any isolated invariant set $S$ admits an index pair. The {\em Conley index} for an isolated invariant set $S$  is defined to be the pointed space
$$I(S) :=(N/L,[L]).$$  The pointed homotopy type of $I(S)$ is an invariant of the triple $(X,\varphi_t,S)$. Moroever, the Conley index is invariant under continuous deformations of the flow, as long as $S$ remains isolated in a suitable sense.

Floer \cite{FloerConley} and Pruszko \cite{Pruszko} developed an equivariant refinement of Conley index theory. If a Lie group $G$ acts smoothly on $M$ preserving the flow $\varphi$ and the set $S$, there exists a $G$-equivariant index pair $(N, L)$ for $S$, and the Conley index $I_G(S)=(N/L, [L])$ is well-defined up to $G$-equivariant homotopy equivalence. Furthermore, it was shown by G{\polhk{e}}ba \cite[Proposition 5.6]{Geba} that one can choose $N$ and $L$ so that $I_G(S)$ is a finite $G$-CW complex. 

Returning to the situation at hand, consider the flow $\varphi^\nu_\tau$ on $\vnt$. Let $S^\nu_\tau$ denote the set of points that lie on the trajectories of $\varphi^\nu_\tau$ inside $\overline{B(2R)}$. Recall from Proposition~\ref{prop:3} that these trajectories stay inside $B(R)$. Therefore, $S^\nu_{\tau}$ is an isolated invariant set, and we can associate to it a $G$-equivariant Conley index:
$$ I^\nu_\tau = I_{G}(S^\nu_\tau).$$

\subsection{Pin(2)-equivariant Seiberg-Witten Floer homology} \label{sec:swfh}
Let us understand to what extent the $G$-homotopy type of the Conley index $I^{\nu}_{\tau}$ depends on the choices made in its construction. From general Conley index theory we know that it is a deformation invariant as long as we do not change the ambient space $\vnt$. Therefore, the only choices we need to consider are $\nu \gg 0$, $\tau \ll 0$, and the Riemannian metric $g$.

It is shown in \cite[Section 7]{Spectrum} that when we increase the upper cut-off $\nu$, the Conley index $I^{\nu}_{\tau}$ is unchanged (up to equivalence). On the other hand, when we decrease the lower cut-off from $\tau$ to $\tau' < \tau$, the Conley index changes by suspension by the $G$-representation $V^\tau_{\tau'}$. We know from Proposition~\ref{prop:susp} that Borel homology is invariant under suspension, up to a shift in grading. It follows that the normalized Borel homology
\begin{equation}
\label{eq:swfhg}
 \swfh^G_*(Y, \s, g) := \tH_{*+ \dim V^0_{\tau}}^G(I^\nu_{\tau}; \F)
 \end{equation}
is an invariant of the triple $(Y, \s, g)$.

The dependence on $g$ was also studied in \cite{Spectrum}. Suppose we deform the metric in a one-parameter family $\{g_t\}_{t \in [0,1]}$, and we choose $\nu \gg 0$ and $\tau \ll 0$ such that they are not  eigenvalues of $\ell=(*d, \dirac)$ at any time during the deformation. (Such choices exist if the deformation is small.) Then the dimension of $\vnt$ does not change, and the properties of the Conley index show that it is invariant under this deformation, up to $G$-equivalence.

Nevertheless, if we have a deformation of $g$ as above, the homology $\swfh^G_*(Y, \s, g)$ from \eqref{eq:swfhg} may change by a shift in degree. This is because the dimension of $V^0_{\tau}$ changes by the spectral flow of the Dirac operator $\dirac$. (By contrast, $*d$ has trivial spectral flow.) The spectral flow of $\dirac$ is controlled by a quantity
$$ n(Y, \s, g) \in \frac{1}{8}\Z,$$ 
which can be defined as follows. (Compare \cite[Section 6]{Spectrum}.) Pick a compact, spin $4$-manifold $(W, \t)$ with boundary $(Y, \s)$. Equip $W$ with a Riemannian metric such that a neighborhood of the boundary is isometric to $[0,1] \times Y$. Let $\Dirac$ be the Dirac operator on $W$ with spectral boundary conditions as in \cite{APS}, and set
\begin{equation}
\label{eq:n}
 n(Y, \s, g) = \ind_{\C}(\Dirac) + \frac{\sigma(W)}{8}.
 \end{equation}
It can be shown that $n(Y, \s, g)$ is independent of the choice of $(W, \t)$, and that during a deformation $\{g_t\}_{t \in [0,1]}$ of the metric on $Y$, the spectral flow of $\dirac$ is given by the formula
$$ \text{S.F.}(\dirac) = n(Y, \s, g_0) - n(Y, \s, g_1).$$

With this in mind, we define
\begin{equation}
\label{eq:swfh}
 \swfh^G_*(Y, \s) := \swfh^G_{*+2n(Y, \s, g)}= \tH_{*+ \dim V^0_{\tau} +2n(Y, \s, g)}^G(I^\nu_{\tau}; \F)
 \end{equation}
to be the {\em $G$-equivariant Seiberg-Witten Floer homology} of $(Y, \s)$. The same arguments as in \cite[proof of Theorem 1]{Spectrum} imply the following: 

\begin{proposition}
\label{prop:swfh}
Let $Y$ be a rational homology $3$-sphere and $\s$ a spin structure on $Y$. The isomorphism class of $\swfh^G_*(Y, \s)$, as a module over $\Ring \cong \F[q, v]/(q^3)$, is an invariant of the pair $(Y, \s)$. 
\end{proposition}

\begin{remark}
The homology $\swfh^G_*$ is the $G$-equivariant analogue of the version $\HMto$ of monopole Floer homology, as defined by Kronheimer and Mrowka \cite{KMbook}. Instead of the Borel homology of the Conley index $I^{\nu}_{\tau}$, one could take its (suitably normalized) co-Borel and Tate homologies, and obtain the $G$-equivariant analogues of $\widehat{\HM}$ and $\overline{\HM}$. 

Since our spectrum-based construction is different from the one in \cite{KMbook}, we have chosen the notation $\swfh$ rather than $\HMto$.
\end{remark}

\subsection{A Pin(2)-equivariant Seiberg-Witten Floer spectrum}
\label{sec:spectrum}
As an aside, in this subsection we explain how the Conley indices $I^{\nu}_{\tau}$ fit naturally into a (metric-dependent) suspension spectrum $\swf(Y, \s, g)$.

Following \cite[Ch. I, \S 2]{LMS}, we define a {\em $G$-universe} $\U$ to be a countably infinite dimensional representation of $G$ with a $G$-invariant inner product, such that:
\begin{itemize}
\item $\U$ contains the trivial representation $\R$, and 
\item $\U$ contains infinitely many copies of each of its finite dimensional subrepresentations.
\end{itemize}

Let $\U$ be a $G$-universe. A {\em $G$-prespectrum} $X$ indexed on $U$ consists of pointed $G$-spaces $X(U)$ for each finite dimensional subrepresentation $U \subset \U$, together with based maps
$$ \sigma^{U'-U}: \Sigma^{U' - U} X(U) \to X(U'),$$
for all $U \subseteq U'$, where $U'-U$ denotes the orthogonal complement of $U$ in $U'$. The maps are required to satisfy an appropriate transitivity condition.

A $G$-prespectrum $X$ is called a {\em $G$-spectrum} if the adjoint maps $X(U) \to \Omega^{U'-U}X(U')$ are homeomorphisms. It is shown in \cite[Ch. I]{LMS} that any $G$-prespectrum can be turned into a $G$-spectrum by a ``spectrification" functor.

Recall from the previous subsections that as we change the value of $\nu \gg 0$, the Conley indices $I^{\nu}_{\tau}$ change by $G$-equivalences, whereas if we change $\tau \ll 0$, they change by suspensions. Let us fix $\nu$ and $\tau$ and consider the universe 
$$\U = V^0_{-\infty} \oplus \R^{\infty}$$ consisting of the eigenspaces of $\ell$ with non-positive eigenvalues (with the $L^2$ inner product), together with infinitely many copies of the trivial representation $\R$. Note that $V^0_{\infty}$ is the direct sum of infinitely many copies of the representations $\R$, $\tR$ and $\H$ of $G=\pin$.

Define a $G$-prespectrum $X=\swfsmall(Y, \s, g, \nu, \tau)$ as the formal desuspension of $I^{\nu}_{\tau}$ by $V^0_{\tau}$, that is:
$$ X(U) = \begin{cases}
\Sigma^{U-V^0_{\tau}} I^\nu_{\tau} & \text{if } V^0_{\tau} \subseteq U,\\
* & \text{otherwise,}
\end{cases}$$
with the maps $\sigma^{U' - U}$ being the obvious identifications when $V^0_{\tau} \subseteq U \subseteq U'$. (Compare \cite[Definition 4.1]{LMS}.) 

We denote by $\swf(Y, \s, g, \nu, \tau)$ the spectrification of $\swfsmall(Y, \s, g, \nu, \tau)$. The arguments in \cite[proof of Theorem 1]{Spectrum} can be used to prove:
\begin{proposition}
The $G$-spectrum $\swf(Y, \s, g) := \swf(Y, \s, g, \nu, \tau)$ is an invariant of the triple $(Y, \s, g)$, up to stable $G$-homotopy equivalence (that is, equivalence in the homotopy category of spectra indexed by $\U$).
\end{proposition}

We can describe the metric-dependent Floer homology $\swfh^G_*(Y, \s, g)$ from \eqref{eq:swfhg} as the $G$-equivariant homology of the spectrum $\swf(Y, \s, g)$, in the sense of \cite{LMM, LMS}. From \eqref{eq:swfh} we deduce that:
$$ \swfh^G_*(Y, \s) = \tH_{*+2n(Y, \s, g)}^G(\swf(Y, \s, g)).$$

As we vary the metric $g$, the universe $\U$ changes, and it is not possible to identify these different universes in a natural way. Nevertheless, by analogy with the construction in \cite[Section 6]{Spectrum}, one could define a metric-independent invariant $\swf(Y, \s)$ that lives in an $G$-equivariant analogue of the classical Spanier-Whitehead category. Thus, one can see that $\swfh^G_*(Y,\s)$ is well-defined as an $\Ring$-module, up to canonical isomorphism. (I would like to thank Mikio Furuta for this observation.)

\subsection{Numerical invariants}
\label{sec:numerical}
Let us return to the $G$-equivariant Seiberg-Witten Floer homology, defined in \eqref{eq:swfh} as the shifted Borel homology of the Conley index $I^{\nu}_{\tau}$. To be able to apply the constructions from Section~\ref{sec:spacesSWF}, we need:

\begin{lemma}
\label{lem:nt}
For all $\nu, - \tau \gg 0$ we can find an index pair $(N, L)$ for $S^\nu_{\tau}$ such that the Conley index $I^{\nu}_{\tau}=N/L$ is a space of type $\swf$ at some level
$$ s \equiv \dim V^0_{\tau} \pmod 4.$$ 
\end{lemma}

\begin{proof}
To understand the action of $G$ on $I^{\nu}_{\tau}$, we use the arguments in \cite[Section 8]{Spectrum}. Precisely, note that the Seiberg-Witten equations have a unique reducible solution $(a, \phi)=(0,0)$. We can perturb the $\csd$ functional by a one-form $\omega \in i\Omega^1(Y)$ to get
$$\csd_\omega(a, \phi) = \csd(a, \phi) + \frac{1}{2} \int_Y a \wedge d\omega.$$
There is still one reducible solution, $(\omega, 0)$, and $\csd_\omega$ evaluates to zero on this solution.

We construct a new isolated invariant set $T=T^\nu_{\tau} \subset \vnt$ using the flow of $\csd_{\omega}$ instead of the flow of $\csd$. Let us interpolate linearly between $0$ and $\omega$, and denote by $\{\ell_{t}\}_{t\in [0,1]}$ the linearizations of the Seiberg-Witten maps on $V$ during this interpolation. If the perturbation $\omega$ is small, we can choose $\nu$ and $\tau$ such that they are not eigenvalues of any $\ell_{t}$ for $t \in [0,1]$. If this is the case, the original Conley index $I^\nu_\tau = I(S^\nu_\tau)$ is $G$-homotopy equivalent to the new Conley index $I(T^\nu_\tau)$.

For a generic choice of $\omega$, we can arrange so that the new reducible solution, $(\omega, 0)$, is a nondegenerate critical point of $\csd_{\omega}|_V$, and such there that are no irreducible critical points $x$ with $\csd_{\omega}(x) \in (0, \epsilon)$, for some fixed $\epsilon > 0$. In this situation, in addition to $T = T^\nu_\tau$, we can identify four other isolated invariant sets in the gradient flow of $\csd_{\omega}|_{\vnt}$:
\begin{itemize}
\item $T^{\irr}_{> 0} = $ the set of (irreducible) critical points $x$ with $\csd_{\omega}(x) > 0$, together with all points on the flow trajectories between the critical points of this type;
\item $T^{\irr}_{\leq 0} = $ same as above, but with $\csd_{\omega}(x) \leq 0$ and requiring $x$ to be irreducible;
\item $T_{\leq 0} = $ same as above, with $\csd_{\omega}(x) \leq 0$ but allowing $x$ to be  reducible or irreducible;
\item $\Theta = \{(\omega, 0)\}$, the reducible critical point by itself.
\end{itemize}

The Conley indices associated to these sets are related to each other by attractor-repeller coexact sequences:
\begin{equation}
\label{eq:ar1}
 I(T_{\leq 0}) \To I(T) \To I(T^{\irr}_{>0}) \To \Sigma I(T_{\leq 0}) \To \dots
 \end{equation}
and 
\begin{equation}
\label{eq:ar2}
 I(T^{\irr}_{\leq 0}) \To I(T_{\leq 0}) \To I(\Theta) \To \Sigma I(T^{\irr}_{\leq 0}) \To \dots
  \end{equation}

The action of $G$ is free in neighborhoods of $T^{\irr}_{> 0}$ and $T^{\irr}_{\leq 0}$, so it is also free on the respective Conley indices (away from the basepoints). Moroever, since the reducible is nondegenerate, we have $$I(\Theta) \cong (V^0_\tau)^+ \cong (\tR^s \oplus \H^p)^+$$
for some $s, p \geq 0$. From the coexact sequences we deduce that the $S^1$-fixed point set of $I(T)$ is $G$-equivalent to the $S^1$-fixed point set of $I(\Theta)$, which is the sphere $(\tR^s)^+$. Also, the action of $G$ must be free on the complement of $I(T)^{S^1}$.

This proves that $I^\nu_\tau=I(S^\nu_\tau) \sim I(T^\nu_\tau)$ is of type $\swf$ for some particular $\nu$ and $\tau$. In turn, it implies the same thing for all $\nu' > \nu$ and $\tau' < \tau$, because the corresponding Conley indices can only change by suspensions by $\tR$ or $\H$ (up to $G$-homotopy equivalence), and the property of being of type $\swf$ is preserved by such suspensions.

Since $V^0_{\tau} \cong \tR^s \oplus \H^p$ and $\dim \H = 4$, we also know that the level $s$ of $I^{\nu}_{\tau}$ is congruent to $\dim V^0_{\tau}$ modulo $4$.
\end{proof}

Recall from Section~\ref{sec:spacesSWF} that to any space $X$ of type $\swf$ at level $s$ one can associate three quantities $a(X), b(X), c(X) \in \Z$, all congruent to $s$ modulo $4$. 

For $\nu$ and $-\tau$ sufficiently large, let us define: 
\begin{align}
 \alpha(Y, \s) &=  \bigl(a(I^{\nu}_{\tau}) - \dim V^0_{\tau}\bigr)/2 - n(Y, \s, g), \label{eq:a}\\
 \beta(Y, \s) &=  \bigl(b(I^{\nu}_{\tau}) - \dim V^0_{\tau} \bigr) / 2 - n(Y, \s, g), \label{eq:b}\\
 \gamma(Y, \s) &= \bigl(c(I^{\nu}_{\tau}) - \dim V^0_{\tau} \bigr)/2 - n(Y, \s, g) . \label{eq:c} 
\end{align}

\begin{proposition}
\label{prop:abc2}
If $Y$ is a rational homology three-sphere, and $\s$ is a spin structure on $Y$, then the quantities $\alpha(Y, \s), \beta(Y, \s), \gamma(Y, \s) \in \tfrac{1}{8} \Z$ are invariants of the pair $(Y, \s)$. Moreover, we have:
$$\alpha(Y, \s) \equiv \beta(Y, \s) \equiv \gamma(Y, \s) \equiv -\mu(Y, \s) \pmod {2\Z},$$
where $\mu(Y, \s)$ is the generalized Rokhlin invariant.
\end{proposition}

\begin{proof}
By analogy with \eqref{eq:iH}, set
$$ \iswfh^G_*(Y, \s) := \bigcap_{l \geq 0} \im \bigl (v^l : {\swfh}^G_{*+4l}(Y, \s) \To {\swfh}^G_*(Y, \s) \bigr).$$

Let $s$ be the level of the Conley index $I^{\nu}_{\tau}$ for some cut-offs $\nu$ and $\tau$. We know from Lemma~\ref{lem:nt} that $s$ is congruent mod $4$ to the dimension of $V^0_{\tau}$. 
Using the descriptions \eqref{eq:ax}, \eqref{eq:bx}, \eqref{eq:cx} of $a$, $b$, $c$, and the description  
\eqref{eq:swfh} of the $G$-equivariant Seiberg-Witten Floer homology, we can write:
\begin{align}
\alpha(Y, \s) &= \tfrac{1}{2} \min \{  r \equiv  - 2n(Y, \s, g) \ (\mod 4\Z) \mid \exists \ x,\ 0 \neq x \in {\iswfh}^G_r(Y, \s)   \}, \label{eq:al} \\
\beta(Y, \s) &= \tfrac{1}{2}\bigl( \min \{  r \equiv - 2n(Y, \s, g)+1\ (\mod 4\Z) \mid \exists \ x,\ 0 \neq x \in {\iswfh}^G_*(Y, \s) \} - 1 \bigr), \label{eq:be} \\
\gamma(Y, \s) &= \tfrac{1}{2}\bigl(   \min \{  r \equiv  - 2n(Y, \s, g)+2\ (\mod 4\Z) \mid \exists \ x, \ 0 \neq x \in {\iswfh}^G_*(Y, \s)  \} - 2 \bigr). \label{eq:ga}
\end{align}

Proposition~\ref{prop:swfh} now implies that $\alpha$, $\beta$ and $\gamma$ are invariants of $(Y, \s)$. 

Recall that $a(I^{\nu}_{\tau}), b(I^{\nu}_{\tau})$ and $c(I^{\nu}_{\tau})$ are all congruent mod $4$ to the level $s$, and hence to the dimension of $V^0_{\tau}$. Looking at the definitions \eqref{eq:a}, \eqref{eq:b}, \eqref{eq:c}, we deduce that $\alpha$, $\beta$ and $\gamma$ are congruent to $-n(Y, \s, g)$ modulo $2\Z$. The quantity $n(Y, \s, g)$ was introduced in \eqref{eq:n} as
$$n(Y, \s, g) = \ind_{\C}(\Dirac) + \frac{\sigma(W)}{8}.$$
Since the Dirac operator $\Dirac$ goes between quaternionic vector spaces, its complex index is divisible by $2$. It follows that  $\alpha(Y, \s)$, $\beta(Y, \s)$ and $\gamma(Y, \s)$ are congruent to $\sigma(W)/{8}$ modulo $2\Z$. The reduction of $\sigma(W)/{8}$ modulo $2\Z$ is exactly the generalized Rokhlin invariant $\mu(Y, \s)$.
\end{proof}

\begin{proposition}
\label{prop:abcMinus}
Let $(Y, \s)$ be an oriented rational homology three-sphere equipped with a spin structure $\s$. Let $-Y$ denote $Y$ with the opposite orientation. Then:
$$ \alpha(-Y, \s) = - \gamma(Y, \s), \ \  \beta(-Y, \s) = -\beta(Y, \s), \ \ \gamma(-Y, \s) = - \alpha(Y, \s).$$
\end{proposition}

\begin{proof}
The Seiberg-Witten flow for $-Y$ is the reverse of the Seiberg-Witten flow for $Y$. Orientation reversal also reverses the signs of the eigenvalues of $\ell$, so if we denote by $\bar V$ the Coulomb slice for $-Y$, then the finite dimensional approximation $V^{\nu}_{\tau}$ can be identified with $\bar V^{-\tau}_{-\nu}$. (Here, we pick $\nu$ and $\tau$ so that they are not eigenvalues of $l$.)
 
Cornea \cite{Cornea} proved that on a stably parallelized manifold (such as $V^\nu_{\tau} \cong \bar V^{-\tau}_{-\nu}$) the Conley indices associated to a flow and its inverse are Spanier-Whitehead dual to each other. His result can be extended to the $G$-equivariant setting. Precisely, one can adapt \cite[proof of Theorem 3.5]{Cornea} to show that we can find index pairs $(N, L)$ and $(N, L')$ for $S^{\nu}_{\tau}$ under the flow $\varphi^\nu_{\tau}$ and its reverse, such that $N, L$ and $L'$ satisfy the hypotheses of Lemma~\ref{lem:Vdual}, with $N$ being embedded in the representation $V^\nu_{\tau}$. It follows that the corresponding Conley indices are equivariantly $(V^{\nu}_{\tau})$-dual. 

The Atiyah-Patodi-Singer index theorem gives  
$$n(Y, \s, g) + n(-Y, \s, g) = \dim_{\C} ( \ker \dirac).$$
Observe also that:
$$ \dim V^0_{\tau} + \dim \bar V^0_{-\nu} + 2\dim_{\C} (\ker \dirac) = \dim V^{\nu}_{\tau}.$$
The desired result now follows from the two equalities above, together with Proposition~\ref{prop:abcDuality} and the formulas \eqref{eq:a}, \eqref{eq:b}, \eqref{eq:c}. 
\end{proof}

\subsection{Behavior under cobordisms} 
 Let $W$ be a compact four-manifold with boundary $Y$, such that $b_1(Y)=0$. Suppose we have a spin structure $\t$ on $W$ whose restriction to $Y$ is $\s$. Following \cite[Section 9]{Spectrum} (as corrected by Khandhawit in \cite{Khandhawit}), we can do finite dimensional approximation for the Seiberg-Witten equations on $W$ with suitable boundary conditions. 
 
We assume for simplicity that $b_1(W)=0$. Pick a Riemannian metric $g$ on $W$ so that the boundary has a neighborhood isometric to $[0,1] \times Y$. Let $\Spin^+$ and $\Spin^-$ denote the two spinor bundles on $W$, and $\Omega^1_g(W)$ denote the space of one-forms on $W$ in double Coulomb gauge, as in \cite[Definition 1]{Khandhawit}. For a fixed cut-off $\nu \gg 0$, we can define a Seiberg-Witten map for $W$:
$$ \widetilde{SW}^{\nu} : i\Omega^1_g(W) \oplus \Gamma(\Spin^+) \To i\Omega^2_+(W) \oplus \Gamma(\Spin^-) \oplus V^{\nu}_{-\infty}.$$
After doing finite dimensional approximation, we obtain from here a based map:
\begin{equation}
\label{eq:Psi}
 \Psi_{\nu, \tau, U, U'}: (U')^+ \To U^+ \wedge I^\nu_{\tau},
 \end{equation}
where $U' \subset i\Omega^1_g(W) \oplus \Gamma(\Spin^+)$ and $U \subset  i\Omega^2_+(W) \oplus \Gamma(\Spin^-)$ are finite dimensional $G$-invariant subspaces. (In fact, $U'$ is determined by $U$, $\nu$ and $\tau$; we refer to \cite[Section 9]{Spectrum} for more details.) As representations of $G$, we have
$$ U' \cong \tR^{m'} \oplus \H^{p'}, \ \ U \cong \tR^m \oplus \H^p,$$
where $m', p', m, p \geq 0$ are related by the formulas:
\begin{align}
 m' - m &= \dim_{\R} V^0_{\tau}(\tR) - b_2^+(W),\label{eq:mm} \\
 p' - p &=  \dim_{\H} V^0_{\tau}(\H) + \ind_{\H}(\Dirac) =  \frac{1}{4} \bigl( \dim_{\R} V^0_{\tau}(\H) + 2n(Y, \s, g) - {\sigma(W)}/4 \bigr). \label{eq:pp}
\end{align}
Here, $V^0_{\tau} = V^0_{\tau}(\tR) \oplus V^0_{\tau}(\H)$ is the decomposition of $V^0_{\tau}$ into its one-form and spinorial parts.

The $S^1$-fixed point set of the Seiberg-Witten map $\widetilde{SW}^{\nu}$ is a linear Fredholm map; see \cite[Proposition 5]{Spectrum}. Starting from here one can identify the $S^1$-fixed point set map of \eqref{eq:Psi}. This is induced on the one-point compactifications by a linear, injective operator
\begin{equation}
\label{eq:PsiFixed}
 \tR^{m'} \To \tR^m \oplus \tR^s,
 \end{equation}
where $s = \dim_{\R} V^0_{\tau}(\tR)$ is the level of $I^{\nu}_{\tau}$ as a space of type $\swf$. (See the proof of Lemma~\ref{lem:nt}.) From \eqref{eq:mm} we have $m'-m=s-b_2^+(W)$, so the cokernel of the map \eqref{eq:PsiFixed} must have dimension $b_2^+(W)$.

The discussion in the previous subsections was phrased in terms of $Y$ being a rational homology $3$-sphere. However, it applies equally well when $Y$ is a disjoint unions of rational homology spheres. (In \cite{Spectrum}, we worked in this greater generality.) The Conley index $I^{\nu}_{\tau}$ coming from a disjoint union is the smash product of the Conley indices coming from each component. 

In particular, suppose we have a compact, spin cobordism $(W, \t)$ between rational homology spheres $Y_0$ and $Y_1$, so that $\del W = (-Y_0) \cup Y_1$. Let $V_0, \bar V_0$ and $V_1$ denote the Coulomb slices corresponding to $Y_0, -Y_0$ and $Y_1$. To simplify notation, let us pick eigenvalue cut-offs $\nu$ and $\tau = - \nu$, so that $(V_0)^{\nu}_{-\nu} \cong (\bar V_0)^{\nu}_{-\nu}.$ Let $(I_0)^{\nu}_{-\nu}, (\bar I_0)^{\nu}_{-\nu}$ and $(I_1)^{\nu}_{-\nu}$ denote the Conley indices associated to the finite-dimensional approximations for the Seiberg-Witten maps on $Y_0, -Y_0$ and $Y_1$, respectively. The map \eqref{eq:Psi} can be written as
\begin{equation}
\label{eq:PsiAgain}
 (U')^+ \To U^+ \wedge (I_1)^\nu_{\tau} \wedge (\bar I_0)^\nu_{-\nu}.
 \end{equation}

Recall from the proof of Proposition~\ref{prop:abcMinus} that the Conley indices $(I_0)^{\nu}_{-\nu}$ and $(\bar I_0)^{\nu}_{-\nu}$ are $(V_0)^{\nu}_{-\nu}$-dual to each other. Thus, we have a duality map
$$ \eps : (\bar I_0)^\nu_{-\nu} \wedge (I_0)^{\nu}_{-\nu} \to ((V_0)^{\nu}_{-\nu})^+.$$
Consider the smash product of the map \eqref{eq:PsiAgain} with the identity on $(I_0)^{\nu}_{-\nu}$, and the smash product of the identity on $U^+ \wedge (I_1)^\nu_{-\nu}$ with the duality map $\eps$. The composition of these two maps takes the form:
$$ (U')^+ \wedge (I_0)^{\nu}_{-\nu} \To U^+ \wedge (I_1)^\nu_{-\nu} \wedge  ((V_0)^{\nu}_{-\nu})^+.$$
After changing the vector spaces involved by isomorphisms, we can write this more simply as a map between suspensions:
\begin{equation}
\label{eq:uu}
f: \Sigma^{m'\tR} \Sigma^{p'\H} (I_0)^{\nu}_{-\nu} \To \Sigma^{m''\tR} \Sigma^{p''\H} (I_1)^{\nu}_{-\nu}. 
\end{equation}

Using \eqref{eq:mm}, \eqref{eq:pp}, we can figure out the differences in suspension indices. Precisely, we have
$$ m' - m'' = \dim_{\R} \bigl( (V_1)^0_{-\nu}(\tR)\bigr) - \dim_{\R} \bigl((V_0)^0_{-\nu}(\tR)\bigr) - b_2^+(W) $$ 
and
$$ p' - p'' =  \frac{1}{4} \Bigl( \dim_{\R} \bigl( (V_1)^0_{-\nu}(\H) \bigr) - \dim_{\R} \bigl( (V_0)^0_{-\nu}(\H) \bigr) + 2n(Y_1, \t|_{Y_1}, g) - 2n(Y_0, \t|_{Y_0}, g) - {\sigma(W)}/4 \Bigr).$$
The $S^1$-fixed point set of \eqref{eq:uu} is still induced (on the one-point compactifications) by a linear injective map with cokernel of dimension $b_2^+(W)$.

\begin{proposition}
\label{prop:cob}
Suppose that $W$ is a smooth, oriented, negative-definite cobordism from $Y_0$ to $Y_1$. Let $b_2(W)$ denote the second Betti number of $W$. If $W$ admits a spin structure $\t$, then:
\begin{align*}
\alpha(Y_1, \t|_{Y_1}) & \geq \alpha(Y_0, \t|_{Y_0}) + \tfrac{1}{8} b_2(W), \\
\beta(Y_1, \t|_{Y_1}) & \geq \beta(Y_0, \t|_{Y_0}) + \tfrac{1}{8} b_2(W), \\
\gamma(Y_1, \t|_{Y_1}) & \geq \gamma(Y_0, \t|_{Y_0}) + \tfrac{1}{8} b_2(W). 
\end{align*}
\end{proposition}

\begin{proof}
By doing surgery on loops in $W$, we can assume without loss of generality that $b_1(W)=0$, so we can apply the discussion in this section. Observe that \eqref{eq:uu} is a $G$-equivariant map between spaces of type $\swf$. Furthermore, since $b_2^+(W)=0$, the associated $S^1$-fixed point set map is a $G$-homotopy equivalence. Thus, the hypotheses of Proposition~\ref{prop:abcIneq} are satisfied. The results follow from that proposition, in view of the formulas \eqref{eq:a}, \eqref{eq:b}, \eqref{eq:c}.
\end{proof}

\begin{corollary}
\label{cor:0}
Suppose that $W$ is a smooth oriented cobordism between rational homology spheres $Y_0$ and $Y_1$, such that $b_2(W)=0$. If $W$ admits a spin structure $\t$, then:
$$
\alpha(Y_0, \t|_{Y_0}) = \alpha(Y_1, \t|_{Y_1}), \ \
\beta(Y_0, \t|_{Y_0}) = \beta(Y_1, \t|_{Y_1}), \ \
\gamma(Y_0, \t|_{Y_0}) = \gamma(Y_1, \t|_{Y_1}). $$
\end{corollary}

\begin{proof}
Apply Proposition~\ref{prop:cob} to both $W$ and $-W$, where the latter is viewed as a cobordism from $Y_1$ to $Y_0$.
\end{proof}

If $Y$ is an integral homology sphere, then it has a unique spin structure $\s$, and we write
$$ \alpha(Y) = \alpha(Y, \s), \ \ \beta(Y) = \beta(Y, \s), \ \ \gamma(Y)=\gamma(Y, \s).$$

Recall from the introduction that $\theta^3_H$ denotes the (integral) homology cobordism group in dimension three. Corollary~\ref{cor:0} implies that our three invariants descend to maps
$$\alpha, \beta, \gamma: \theta_3^H \to \Z.$$
(We do not claim that these maps are homomorphisms.)

We can now complete the proofs of the results announced in the introduction:

\begin{proof}[Proof of Theorem~\ref{thm:main}] This follows directly from Propositions~\ref{prop:abc2}, \ref{prop:abcMinus} and \ref{prop:cob}.
\end{proof}

\begin{proof}[Proof of Corollary~\ref{cor:no2}] Observe that $Y \# Y$ being homology cobordant to $S^3$ is the same as $Y$ being homology cobordant to $-Y$. If $Y$ has this property, the properties listed in Theorem~\ref{thm:main} imply that $\beta(Y) = \beta(-Y) =- \beta(Y)$, so $\beta(Y)=0$ and hence $\mu(Y)=0$.
\end{proof}

\begin{remark}
When we proved invariance of $\alpha, \beta, \gamma$ in Proposition~\ref{prop:abc2}, we used the invariance of the isomorphism class of $\swfh^G_*(Y, \s)$ (cf. Proposition~\ref{prop:swfh}). In turn, this relied on the arguments from \cite{Spectrum} about the behavior of the Conley index under changes in the eigenvalue cut-off and the Riemannian metric. It is worth pointing out that one can give an alternative proof of the invariance of $\alpha, \beta, \gamma$ (and thus establish Theorem~\ref{thm:main}) by using Corollary~\ref{cor:0} instead. Indeed, suppose that, a priori, $\alpha, \beta$ and $\gamma$ depended on the metric on $Y$. Let $g_0$ and $g_1$ be two Riemannian metrics on $Y$, and consider the identity cobordism $I \times Y$ with a metric that restricts to $g_0$ on one end and $g_1$ on the other, and is cylindrical near the ends. Then, the argument in Corollary~\ref{cor:0} shows that $\alpha, \beta, \gamma$ are metric independent.
\end{remark}

\subsection{A fourth invariant}
Let $Y$, $\s$, $g$, $\nu$ and $\tau$ be as in Section~\ref{sec:swfh}. Instead of considering the $\pin$-equivariant homology of the Conley index $I^{\nu}_{\tau}$ as in \eqref{eq:swfh}, one can take its suitably normalized $S^1$-equivariant homology, with coefficients in an arbitrary field $\f$ (of characteristic $p$). Set: 
$$ \swfh^{S^1}_*(Y, \s; \f) := \tH_{*+ \dim V^0_{\tau} +2n(Y, \s, g)}^{S^1}(I^\nu_{\tau}; \f)$$
This is the {\em $S^1$-equivariant Seiberg-Witten Floer homology} of $(Y, \s)$. It can be viewed as the Borel homology of the $S^1$-equivariant suspension spectrum $\swf(Y, \s)$ constructed in \cite{Spectrum}.

In Section~\ref{sec:fourth} we defined a quantity $d_p$ associated to a space of type $\swf$. Let us apply this to the Conley index $I^{\nu}_{\tau}$. After a suitable normalization, it yields an invariant
$$ \delta_p(Y, \s) =   \bigl(d_p(I^{\nu}_{\tau}) - \dim V^0_{\tau}\bigr)/2 - n(Y, \s, g) \in \tfrac{1}{8}\Z.$$

Alternately, if we set 
$$ \iswfh^{S^1}_*(Y, \s; \f) := \bigcap_{l \geq 0} \im \bigl (U^l : {\swfh}^{S^1}_{*+2l}(Y, \s; \f) \To {\swfh}^{S^1}_*(Y, \s; \f) \bigr),$$
we can write:
$$ \delta_p(Y, \s) = \tfrac{1}{2}\min \{  r \mid \exists \ x,\ 0 \neq x \in {\iswfh}^{S^1}_r(Y, \s; \f)   \} .$$

It is customary to take $\f=\R$, so $p=0$. In this case we write $\delta=\delta_0$. After a change in sign, the invariant $\delta$ is the exact analogue of the correction term defined by Fr{\o}yshov in \cite{FroyshovHM} and Kronheimer-Mrowka \cite[Section 39.1]{KMbook}. In Heegaard Floer theory, the  counterpart of $\delta$ is one-half of the correction term $d(Y, \s)$ defined by Ozsv\'ath and Szab\'o in \cite{AbsGraded}.

\begin{remark}
Although in principle the correction terms (in monopole Floer, or Heegaard Floer theory) depend on the characteristic $p$ of the underlying field, in practice no examples of $3$-manifolds are known where this makes a difference.
\end{remark}

\subsection{Examples}
The simplest case where Seiberg-Witten Floer homology can be computed is that of elliptic rational homology $3$-spheres (quotients of $S^3$). If $Y$ is such a manifold, then it admits a metric $g$ with positive scalar curvature, so by the arguments in \cite[Section 10]{Spectrum} or \cite[Section 7.1]{GluingBF}, the Conley index $I^{\nu}_{\tau}$ is a representation sphere, and we get that
$$ \swfh^G_*(Y, \s) \cong \F[q,v]/(q^3),$$
shifted in degree by $-2n(Y, \s, g)$. The same is true for the $S^1$-equivariant Seiberg-Witten Floer homology. Therefore, 
$$ \alpha(Y, \s) = \beta(Y, \s) = \gamma(Y, s) = \delta(Y, \s) = -n(Y, \s, g).$$
In particular, for $Y=S^3$ we obtain
$$ \alpha(S^3) = \beta(S^3) = \gamma(S^3) = 0.$$

Next, let us consider the Brieskorn spheres $\Sigma(2,3,n)$ with $\gcd (6, n) =1$, oriented as boundaries of negative definite plumbings. With these conventions, $\Sigma(2,3,6m-1)$ is $-1/m$ surgery on the left-handed trefoil, and $\Sigma(2,3,6m+1)$ is $-1/m$ surgery on the right-handed trefoil. 

The $S^1$-equivariant Floer spectrum $\swf(-\Sigma(2,3,n))$ was computed in \cite[Section 7.2]{GluingBF}; the Floer spectrum for $\Sigma(2,3,n)$ is its Spanier-Whitehead dual. The calculation was based on the work of Mrowka, Ozsv\'ath and Yu \cite{MOY}, who described the Seiberg-Witten solutions and flow trajectories for Seifert fibrations equipped with particular Riemannian metrics. (They also used some nonstandard connections instead of the Levi-Civita connections, but finite-dimensional approximation still works in this setting, and yields the same answers.) One then uses the attractor-repeller sequences \eqref{eq:ar1}, \eqref{eq:ar2} to obtain information about the Floer spectrum. The same methods as in \cite[Section 7.2]{GluingBF} can be employed to compute the $G$-equivariant Seiberg-Witten Floer homology; we just have to keep track of the additional symmetry when describing the Seiberg-Witten flow. 

When $n=12k-1$, the Seiberg-Witten equations on $\Sigma(2,3,12k-1)$ have one reducible solution in degree zero, and $2k$ irreducibles in degree one. The irreducibles come in $k$ pairs related by the action of $j \in G.$ Each irreducible is connected to the reducible by a single flow line. From \eqref{eq:ar1}, \eqref{eq:ar2} we deduce that there is a long exact sequences on Borel homology:
$$ \dots \To H_*(BG; \F) \To \swfh^G_*(\Sigma(2,3,12k-1)) \To (\F^k)_{[1]} \To \dots $$
where the subscript $[1]$ denotes the respective degree. The same discussion as in \cite[Section 7.2]{GluingBF} shows that the connecting map from $(\F^k)_{[1]}$ to $H_0(BG; \F) \cong \F$ must be non-trivial. This implies that, as a module, the $G$-equivariant Seiberg-Witten Floer homology of $\Sigma(2,3,12k-1)$ is
\[ \xymatrixcolsep{.7pc}
\xymatrixrowsep{0pc}
\xymatrix{
\ \ \ \F^{k-1} & & & & & & & & & \\
\oplus & & & & & & & & & \\
   \F  &  \F \ar@/_1pc/[l]_{q} & 0 & \F  & \F \ar@/_1pc/[l]_{q} \ar@/^1pc/[llll]^{v} & \F \ar@/_1pc/[l]_{q} \ar@/^1pc/[llll]^{v} & 0 & \dots  \ar@/^1pc/[llll]^{v} & \dots \ar@/^1pc/[llll]^{v} & \dots \ar@/^1pc/[llll]^{v} 
} \]
in degrees $1, 2, 3, \dots$ In view of the formulas \eqref{eq:al}, \eqref{eq:be}, \eqref{eq:ga}, we have
$$ \alpha(\Sigma(2,3,12k-1)) = 2, \ \ \beta(\Sigma(2,3,12k-1))= \gamma(\Sigma(2,3,12k-1)) = 0.$$

In particular, $\swfh^G_*(\Sigma(2,3,11))$ agrees with the Borel homology of the unreduced suspension $\tilde{G}$ of $G$, which was discussed in Example~\ref{ex:G}. We should think of $\tilde{G}$ as a model for the $G$-equivariant Floer spectrum of $\Sigma(2,3,11)$. In fact, if we equip $\Sigma(2,3,11)$ with the metric from \cite{MOY}, we can apply the methods in \cite{GluingBF} to show that (for $\nu, -\tau \gg 0$) the Conley indices $I^{\nu}_{\tau}$ are $G$-equivalent to suitable suspensions of $\tilde{G}$.

Next, let us consider the case $n=12k-5$. This is entirely similar to $n=12k-1$, except that the reducible is in degree $-2$ and the irreducibles in degree $-1$. The $G$-equivariant Seiberg-Witten Floer homology agrees to the one for $\Sigma(2,3,12k-1)$, shifted in degree by $-2$. Thus, we have
$$ \alpha(\Sigma(2,3,12k-5)) = 1, \ \ \beta(\Sigma(2,3,12k-5))= \gamma(\Sigma(2,3,12k-5)) = -1.$$

When $n=12k+1$, there is one reducible in degree $0$ and $2k$ irreducibles in degree $-1$. As before, the irreducibles come in $k$ pairs related by the action of $j$. We find a long exact sequence
$$ \dots \To H_*(BG; \F) \To \swfh^G_*(\Sigma(2,3,12k+1)) \To (\F^k)_{[-1]} \To \dots $$

For grading reasons, the only possibility is that $\swfh^G_*(\Sigma(2,3,12k+1))$ is (even as a module) isomorphic to the direct sum $H_*(BG; \F) \oplus (\F^k)_{[-1]}$. From here we get
$$ \alpha(\Sigma(2,3,12k+1)) = \beta(\Sigma(2,3,12k+1))= \gamma(\Sigma(2,3,12k+1)) = 0.$$

The case when $n=12k+5$ is similar, except for a degree shift by $-2$ in homology. Therefore,
$$ \alpha(\Sigma(2,3,12k+5)) = \beta(\Sigma(2,3,12k+5))= \gamma(\Sigma(2,3,12k+5)) = 1.$$

We summarize our results in the following table. For comparison, we have also included the corresponding values of the correction term $\delta$ in $S^1$-equivariant Seiberg-Witten Floer homology, and of the Casson invariant $\lambda$:

\medskip
\begin{center}
\begin{tabular}{| c | |  c | c | c | c | c | c |}
\hline
Brieskorn sphere & {$\alpha$} & {$\beta$} & {$\gamma$} & {$\delta = d/2 = -h$} & $\lambda$ \\
\hline
\hline
$\Sigma(2,3,12k-5)$  & $1$ & $-1$ & $-1$ & $0$ & $-2k+1$ \\
\hline
$\Sigma(2,3,12k-1)$  & $2$ & $0$ & $0$ & $1$ & $-2k$ \\
\hline
$\Sigma(2,3,12k+1)$  & $0$ & $0$ & $0$ & $0$ & $-2k$ \\
\hline
$\Sigma(2,3,12k+5)$  & $1$ & $1$ & $1$ & $1$ & $-2k-1$\\
\hline
\end{tabular}
\end{center} 
\medskip

Here, the values of $\delta$ for these Brieskorn spheres can be readily deduced from \cite[Section 7.2]{GluingBF}. They agree (up to a sign) with the values of the Fr{\o}yshov invariant $h$. The latter can be computed using the surgery exact triangles in monopole Floer homology \cite{KMOS}, along the lines of the corresponding computation for the correction terms $d$ in Heegaard Floer homology  \cite[Section 8.1]{AbsGraded}. Finally, the values of the Casson invariant \cite{Casson} can be deduced from its surgery formula applied to the two trefoils. 

\section{Further directions}
Our construction of $\pin$-equivariant Seiberg-Witten Floer homology was limited to rational homology spheres. This is because the Seiberg-Witten configuration space acquires non-trivial topology when $b_1(Y) > 0$, and doing finite dimensional approximation in this setting becomes more difficult. Nevertheless, as mentioned in the introduction, we expect that one can define $\pin$-Floer homologies  for all compact $3$-manifolds equipped with spin structures. Both monopole Floer homology, as constructed by Kronheimer and Mrowka in their book \cite{KMbook}, and the Heegaard Floer homology of Ozsv\'ath and Szab\'o \cite{HolDisk, HolDiskTwo, HolDiskFour} are defined for arbitrary $3$-manifolds. We expect that one can construct $\pin$-versions of these theories. Moreover, if one were to establish $\pin$-versions of the usual surgery exact triangles, this should allow the computation of the invariants $\alpha, \beta, \gamma$ in more examples.

In particular, it would be interesting to compute our invariants for a larger class of plumbed $3$-manifolds, as was done for Heegaard Floer homology in \cite{Plumbed}. Neumann and Siebenmann \cite{NeumannPlumbed, SiebenmannPlumbed} independently constructed an invariant $\bar \mu(Y, \s) \in \tfrac{1}{8}\Z$ for spin $3$-manifolds that are given by plumbing spheres along a tree. The Neumann-Siebenmann invariant reduces to the generalized Rokhlin invariant mod $2$. (We follow the conventions for $\bar \mu$ from \cite{Saveliev}, which differ from the original ones in \cite{NeumannPlumbed} by a factor of $1/8$.) Neumann \cite{NeumannPlumbed} conjectured that $\bar \mu$ is a homology cobordism invariant, and Saveliev showed this to be true when we restrict $\bar \mu$ to the class of Seifert fibered integral homology spheres \cite{SavelievFF}; see also \cite{FukumotoFuruta, FukumotoFurutaUe}. We propose the following:

\begin{conjecture}
If $(Y, \s)$ is a Seifert fibered rational homology $3$-sphere with a spin structure, then $\beta(Y, \s)= -\bar\mu(Y, \s)$. 
\end{conjecture}

The conjecture holds for the Brieskorn spheres $\Sigma(2,3,6m \pm 1)$ considered in the last section. 

A related question is the existence of an analogue of $\alpha$, $\beta$ or $\gamma$ in instanton theory. Let $I_*$ denote the original instanton homology defined by Floer \cite{Floer}. Saveliev \cite{SavelievInst, Saveliev} conjectured that the quantity
$$ \nu(Y) = \frac{1}{2}\sum_{n=0}^7 (-1)^{(n+1)(n+2)/2} \dim_{\Q} I_n(Y)$$
is an invariant of homology cobordism, and showed that $\nu(Y) = \bar \mu(Y)$ for all Seifert fibered homology spheres. It is possible that the invariant $\nu$ (perhaps defined with the field $\Q$ replaced by $\F$) is related to $-\beta$.

When $Y$ is a Seifert fibered homology sphere, $\nu(Y)=\bar \mu(Y)$ can also be interpreted as half the Lefschetz number of the map induced on $I_*(Y)$ by the mapping cylinder of a canonical involution on $Y$; see \cite{RubermanSaveliev}. A related interpretation exists in the context of Seiberg-Witten theory \cite[Section 11.3]{MRS}, using a Casson-type invariant $\lambda_{\operatorname{SW}}$ for $4$-manifolds with the homology of $S^1 \times S^3$. One can ask about the relation between $\lambda_{\operatorname{SW}}$ and the invariants constructed in this paper.

In yet another direction, it would be interesting to understand the behavior of $\pin$-equivariant Floer homology (and of the invariants $\alpha$, $\beta$, $\gamma$) under taking connected sums. We expect that the connected sum of two $3$-manifolds corresponds to the smash product of their Floer spectra. The Borel homology of a smash product is related to the Borel homology of the two pieces by an Eilenberg-Moore spectral sequence. The possible existence of non-trivial higher differentials makes the behavior of $\alpha, \beta$ and $\gamma$ difficult to predict.

Finally, we mention that one can turn invariants of homology cobordism into invariants of (smooth) knot concordance. Given a knot $K \subset S^3$, one simply evaluates the original invariant to the double cover of $S^3$ branched along $K$. This was done for the Ozsv\'ath-Szab\'o correction term $d$ in \cite{MOwens}. It would be worthwhile to study the concordance invariants associated to $\alpha$, $\beta$, and $\gamma$.

\bibliography{biblio}
\bibliographystyle{custom}

\end{document}